\documentclass[12pt]{article}
\usepackage{}

\usepackage{epsfig}
\usepackage{latexsym}
\usepackage{caption}
\usepackage{amsfonts}

\usepackage{amssymb}
\usepackage{mathrsfs}
\usepackage{amsmath}
\usepackage{enumerate}
\usepackage{graphicx}
\usepackage{MnSymbol}
\usepackage{float}
\usepackage{pict2e}
\usepackage{enumerate}
\usepackage{enumitem}
\usepackage{booktabs} 

\usepackage{graphics}
\usepackage{MnSymbol}
\usepackage{tikz}
\usepackage{cite}
\usepackage{ulem}
\usepackage{extarrows}
\usepackage{cases}
\usepackage{tabu}

\usepackage{tikz-3dplot}
\usepackage{asymptote}

\usepackage{soul,color} 
\definecolor{lightgreen}{RGB}{176,220,213}

\definecolor{brightred}{RGB}{255,80,102}
 
\soulregister\cite7
\soulregister\ref7
\soulregister\pageref7

\usepackage[colorinlistoftodos,prependcaption,textsize=scriptsize,color=olive!70, textwidth=30mm,]{todonotes}

\numberwithin{equation}{section}
\usepackage[pdfauthor={derajan},pdftitle={How to do this},pdfstartview=XYZ,bookmarks=true,
colorlinks=true,linkcolor=blue,urlcolor=blue,citecolor=blue,pdftex,bookmarks=true,linktocpage=true,  hyperindex=true]{hyperref}

\usepackage{color}

\def \bl  {\textcolor{blue} }

\makeatletter

\newcommand{\Rmnum}[1]{\expandafter\@slowromancap\romannumeral #1@}

\makeatother

\begin{document}

\newtheorem{theorem}{Theorem}[section]
\newtheorem{observation}[theorem]{Observation}
\newtheorem{corollary}[theorem]{Corollary}
\newtheorem{algorithm}[theorem]{Algorithm}
\newtheorem{definition}[theorem]{Definition}
\newtheorem{guess}[theorem]{Conjecture}
\newtheorem{claim}{Claim}[section]
\newtheorem{problem}[theorem]{Problem}
\newtheorem{question}[theorem]{Question}
\newtheorem{lemma}[theorem]{Lemma}
\newtheorem{proposition}[theorem]{Proposition}
\newtheorem{fact}[theorem]{Fact}

\captionsetup[figure]{labelfont={bf},name={Fig.},labelsep=period}
\makeatletter
  \newcommand\figcaption{\def\@captype{figure}\caption}
  \newcommand\tabcaption{\def\@captype{table}\caption}
\makeatother

\newtheorem{acknowledgement}[theorem]{Acknowledgement}

\newtheorem{axiom}[theorem]{Axiom}
\newtheorem{case}[theorem]{Case}
\newtheorem{conclusion}[theorem]{Conclusion}
\newtheorem{condition}[theorem]{Condition}
\newtheorem{conjecture}[theorem]{Conjecture}
\newtheorem{criterion}[theorem]{Criterion}
\newtheorem{example}[theorem]{Example}
\newtheorem{exercise}[theorem]{Exercise}
\newtheorem{notation}{Notation}
\newtheorem{solution}[theorem]{Solution}
\newtheorem{summary}[theorem]{Summary}

\newenvironment{proof}{\noindent {\bf
Proof.}}{\rule{3mm}{3mm}\par\medskip}
\newcommand{\remark}{\medskip\par\noindent {\bf Remark.~~}}
\newcommand{\pp}{{\it p.}}
\newcommand{\de}{\em}
\newcommand{\mad}{\rm mad}

\newcommand{\wrt}{with respect to }
\newcommand{\qed}{\hfill\rule{0.5em}{0.809em}}
\newcommand{\var}{\vartriangle}
 
\title{{\large \bf The strong fractional choice number of $3$-choice critical graphs}}
 
\author{Rongxing Xu\thanks{Department of Mathematics, Zhejiang Normal University,  China. E-mail:
xurongxing@yeah.net. Grant Number: NSFC 11871439. Supported also by
China Scholarship Council .}, \and  Xuding Zhu\thanks{Department of Mathematics, Zhejiang Normal University,  China.  E-mail: xdzhu@zjnu.edu.cn. Grant Number: NSFC 11971438. Supported also by
the 111 project of The Ministry of Education of China.}}

\maketitle

\begin{abstract}
A graph $G$ is called $3$-choice critical if $G$ is not $2$-choosable but any proper subgraph  is $2$-choosable. A graph $G$ is strongly fractional $r$-choosable if $G$ is $(a,b)$-choosable for all positive integers $a,b$ for which $a/b \ge r$. The strong fractional choice number of $G$ is $ch_f^s(G) = \inf \{r: G $ is strongly fractional $r$-choosable$\}$. This paper determines the strong fractional choice number of all $3$-choice critical graphs.
\end{abstract}

\section{Introduction}

An {\em $a$-list assignment} of a graph $G$ is a mapping $L$ which assigns to each vertex $v$ of $G$ a set $L(v)$ of $a$ colours. A {\em $b$-fold coloring} of $G$ is a  mapping $\phi$  which assigns to each vertex $v$ of $G$ a  set $\phi(v)$ of $b$ colors such that for every edge $uv$, $\phi(u) \cap \phi(v) = \emptyset$.
An{ \em $(L,b)$-colouring} of $G$ is a $b$-fold coloring $\phi$ of $G$ such that $\phi(v) \subseteq L(v)$  for each vertex $v$.
We say $G$ is {\em $(a,b)$-choosable}  if for any $a$-list assignment $L$ of $G$, there is an $(L,b)$-colouring of $G$, and   $G$ is $(a,b)$-colourable if there is a $b$-fold colouring $\phi$ of $G$ such that $\phi(v) \subseteq \{1,2,\ldots, a\}$ for each vertex $v$.
We say $G$ is {\em  $a$-choosable } (respectively, $a$-colourable) if $G$ is $(a,1)$-choosable  (respectively, $(a,1)$-colourable). The {\em choice number} $ch(G)$ of $G$ is the minimum integer $a$ such that $G$ is $a$-choosable, and the {\em chromatic number} $\chi(G)$ of $G$ is 
the minimum integer $a$ such that $G$ is $a$-colourable.
The concept of list colouring of graphs was introduced independently by Erd\H{o}s, Rubin and Taylor \cite{ERT} and Vizing \cite{Vizing1976} in the 1970's, and has been studied extensively in the literature.

 The {\em fractional chromatic number} $\chi_f(G)$ of a graph $G$ is defined as 
$$\chi_f(G) = \inf \{\frac{a}{b}: G \text{ is $(a,b)$-colourable}.\}$$
The {\em fractional choice number} $ch_f(G)$ of a graph $G$ is defined as 
$$ch_f(G) = \inf \{\frac{a}{b}: G \text{ is $(a,b)$-choosable}.\}$$
It follows from the definition that for any graph $G$, $\chi(G) \le ch(G)$ and $\chi_f(G) \le ch_f(G)$. It is known that there are bipartite graph with arbitrary large choice number. On the other hand, it was proved by Alon, Tuza and Voigt \cite{ATV} that $ch_f(G) = \chi_f(G)$ for every graph $G$. So $ch_f(G)$ is not really a new graph parameter. In particular, $ch_f(G)=2$ for all bipartite graph with at least one edge. 

The concept of strong fractional choice number of a graph was introduced in \cite{Zhu2018}. Given a real number $r$, we say a graph $G$ is {\em strongly fractional $r$-choosable} if $G$ is $(a,b)$-choosable for any  $a, b$ for which $\frac{a}{b} \ge r$.  The {\em strong fractional choice number} $ch_f^s(G)$ of $G$ is defined as 
$$ch_f^s(G) = \inf\{r: G \text{ is strongly fractional $r$-choosable} \}.$$
It follows from the definition that $ch_f^s(G) \ge ch(G)-1$. It was proved in \cite{XuZhu20-concept} that for any finite graph $ch_f^s(G)$ is a rational number and  either $ch_f^s(G)=\chi_f(G)$  or 
the infimum in the definition is attained  and hence can be replaced by the minimum. However, the result in  this  paper shows that if $ch_f^s(G)=\chi_f(G)$, then the infimum in the definition maybe not attained. The parameter $ch_f^s(G)$ may serve as a refinement for the choice number of $G$ and has been studied in a few papers \cite{JZ2019,LZ2019}. However, it remains an open question whether $ch_f^s(G) \le ch(G)$ for every graph $G$.  

For any graph $G$, we have $ch_f^s(G) \ge \chi_f(G)$, and  $ch_f^s(G) \ge 2$ for every graph with at least one edge. It seems to be a difficult problem to characterize all graphs $G$ with $ch_f^s(G)=2$.

Erd\H{o}s, Rubin and Taylor \cite{ERT} characterized all the $2$-choosable graphs. Given a graph $G$, the {\em core} of $G$ is obtained from $G$ by repeatedly removing degree $1$ vertices. Denote by $\Theta_{k_1,k_2,\ldots, k_q}$   the graph consisting of internally vertex disjoint paths of lengths $k_1,k_2, \ldots, k_q$ connecting two vertices $u$ and $v$. Erd\H{o}s, Rubin and Taylor proved that a graph $G$ is $2$-choosable if and only if the core of $G$ is $K_1$ or an even cycle or $\Theta_{2,2,2p}$ for some positive integer $p$. 

We say a graph $G$ is   {\em  $3$-choice critical}  if $G$ is not $2$-choosable but any proper subgraph of $G$ is $2$-choosable.
Voigt characterized all the $3$-choice critical graphs.

\begin{theorem}[\cite{Voigt1998}]
\label{3-choice critical}
	A graph is 3-choice critical if and only if it is one of the following:
	\begin{enumerate}
		\item An odd cycle.
		\item Two vertex-disjoint even cycles joined by a path.
		\item Two even cycles with one vertex in common. 
		\item  $\Theta_{2r,2s,2t}$  with $r\geq 1$, and $s, t>1$, or  
			  $\Theta_{2r+1,2s+1,2t+1}$ with $r\ge 0$, $s,t > 0$.
		\item $\Theta_{2,2,2,2t}$ graph with $t\geq 1$.
	\end{enumerate}
\end{theorem}

 The strong fractional choice numbers of odd cycles  are easily  determined. 
\begin{proposition}
	For odd cycle $C_{2k+1}$, $ch^s_{f}(C_{2k+1})=2+\frac{1}{k}$.
\end{proposition}
\begin{proof}
  It is well-known that $\chi_f(C_{2k+1})=2+\frac 1k$. As $ch^s_f(G) \geq \chi_f(G)$ for any graph $G$, it suffices to show that
  $ch^s_f(C_{2k+1}) \leq 2 + \frac{1}{k}$.
  We shall show that for any $a/b \ge 2+1/k$, $C_{2k+1}$ is $(a,b)$-choosable.  
  
  Assume the vertices of $C_{2k+1}$
  are $(v_0, v_1, \ldots, v_{2k})$ in this cyclic order, $a/b \ge 2+1/k$ and $L$ is an $a$-list assignment of $C_{2k+1}$.   Assume $\bigcup^{2k+1}_{i=1}L(v_i) = \{  c_1, c_2, \ldots, c_p\}$.
  By  permuting   colours, we may assume that 
  $\bigcap^{2k+1}_{i=1}L(v_i) = \{  c_1, c_2, \ldots, c_q\}$, where $0 \le q \le a \le p$. (Note that   $q=0$ when $\bigcap^{2k+1}_{i=1}L(v_i) = \emptyset$).
   For $i=q+1,q+2,\ldots, p$, let $s_i$ be an arbitrary index such that $c_i \not\in L(v_{s_i})$.  We recursively assign colours $c_1, c_2, \ldots, c_p$ to vertices of $C_{2k+1}$.
   Assume colours  $c_1,c_2,\ldots, c_{i-1} $ have been assigned to vertices of $C_{2k+1}$ already. We assign colour $c_i$ to vertices of $C_{2k+1}$ as follows:
 
    If $i \le q$, then assign colour $c_i$ to vertices in the set  $  \{v_i,v_{i+2}, \ldots, v_{i+2k-2}\}$, where the summations in the indices are modulo $2k+1$.     
  
  If $i \ge q+1$, then we traverse the vertices of $C_{2k+1}$ one by one in the order $v_{s_i}, v_{s_i+1}, \ldots, v_{s_i+2k}$, and assign colour $i$ to vertex $v_j$   provided the following hold:
  \begin{itemize}
  	\item  $c_i \in L(v_j)$ and $c_i$ is not assigned to $v_{j-1}$.  
  	\item $v_j$ has received less than $b$ colours from $c_1,c_2,\ldots, c_{i-1}$.
  \end{itemize}
  
  It follows from the construction that each colour class is an independent set and each vertex $v_j$ is assigned at most $b$ colours and all the colours assigned to $v_j$ are from $L(v_j)$. Now we show that   each vertex is assigned exactly $b$ colours. 
  
  Assume to the contrary that $v_j$ is assigned at most  $b-1$ colours.  Assume $c_i \in L(v_j)$ and $c_i$ is not assigned to $v_j$. It follows from the colouring procedure that one of the following holds:
  \begin{enumerate}
  	\item $c_i$ is assigned to $v_{j-1}$.
  	\item $i \leq q$ and $j=i+2k$.
  \end{enumerate}   
  The first case occurs at most $b$ times as $v_{j-1}$ receives at most $b$ colours, and the second case occurs at most $ \lceil \frac{q}{2k+1} \rceil \le \lceil \frac{a}{2k+1} \rceil$ times.
   Therefore,
   $$a=|L(v_j)| \le b-1 +b+ \lceil \frac{a}{2k+1} \rceil < 2b + \frac{a}{2k+1}$$
   and hence 
    $a/b < 2+1/k$, contrary to our assumption. 
\end{proof}

 The main result of this paper is that every bipartite 3-choice critical graphs has strong fractional choice number $2$. It suffices to show that every bipartite 3-choice critical graph is $(2m+1, m)$-choosable for any positive integer $m$.
 
 It is known \cite{TuzaVoigt1996-2choosable,Voigt1998} that for an odd integer $m$,  a graph $G$ is   $(2m,m)$-choosable if and only if $G$ is $2$-choosable.  
In \cite{Voigt1998}, Voigt conjectured that every bipartite $3$-choice critical graph $G$  is $(2m,m)$-choosable for every even integer $m$.
If the conjecture were true, then all bipartite 3-choice critical graphs have strong fractional choice number $2$.  The conjecture was verified for $G=\Theta_{2,2,2,2}$ \cite{TuzaVoigt1996}.
However,  Meng, Puleo and Zhu \cite{4choose2} proved that  if $min\{r,s,t\} \geq 3$,  $r,s,t$ have the same parity, then $\Theta_{r,s,t}$ is not $(4,2)$-choosable, and if $t \geq 2$, then $\Theta_{2,2,2,2t}$ is not $(4,2)$-choosable.
Nevertheless, the other bipartite 3-choice critical graphs, i.e., two vertex-disjoint even cycles joined by a path, two even cycles with one vertex in common, $\Theta_{2,2s,2t}$  with $s, t>1$, and $\Theta_{1,2s+1,2t+1}$ with $s,t > 0$, are $(4,2)$-choosable  \cite{4choose2}.
Xu and Zhu \cite{XuZhu2020} strengthened these results and proved that these   graphs are also $(4m,2m)$-choosable for all integer $m$. Note that 
if a graph $G$ is $(4m,2m)$-choosable, then 
it is $(4m-1, 2m-1)$-choosable: if $L$ is a $(4m-1)$-list assignment, then let $c$ be a new colour, and let $L'(v)=L(v) \cup \{c\}$, we obtain a $4m$-list assignment. Let $f$ be a $2m$-fold $L'$-colouring of $G$, and let $g(v)=f(v)-\{c\}$  if $c \in f(v)$ and $g(v)=f(v)-\{c'\}$ if $c \notin f(v)$, where $c'$ is an arbitrary colour in $f(v)$. Then $g$ is a $(2m-1)$-fold $L$-colouring of $G$. 
So   $ch_f^s(G)=2$ if one of the following holds:
\begin{enumerate}
	\item $G$ is two vertex-disjoint even cycles joined by a path, two even cycles with one vertex in common.
	\item $G=\Theta_{2,2s,2t}$  with $s, t > 1$.   
	\item $G =\Theta_{1,2s+1,2t+1}$ with $s,t \ge 1$.
 \item $G=\Theta_{2,2,2,2}$. 
\end{enumerate}  

In this paper, we prove the following result.
 
\begin{theorem}
	\label{thm-main}
	If $G=\Theta_{2r,2s,2t}$  with $r,s, t>1$, or  
	$G=\Theta_{2r+1,2s+1,2t+1}$ with $r,s,t > 0$,  or  $G=\Theta_{2,2,2,2t}$ graph with  $t\geq 1$, then $G$ is $(2m+1,m)$-choosable for any positive integer $m$.
\end{theorem}   

Thus if $G$ is a bipartite 3-choice critical graph, then for any $r > 2$, $G$ is strongly fractional $r$-choosable. 
Hence   we have the following corollary.

\begin{corollary}
	 Every bipartite 3-choice critical graph $G$ has $ch_f^s(G)=2$.
\end{corollary}  
 
\section{Preliminaries}

The proof of Theorem \ref{thm-main} uses the idea   in \cite{4choose2,XuZhu2020}: Assume $G$ is a graph as in Theorem \ref{thm-main} and $L$ is a  $(2m+1)$-list assignment  of $G$.  Let $u,v$ be the two vertices of $G$ of degree at least $3$. Then $G-\{u,v\}$ is the disjoint union of a family of  three or four paths,  where each end vertex of these paths has exactly one neighbour in $\{u,v\}$ unless the path consists of a single vertex $w$, in which case $w$ is adjacent to both $u$ and $v$. 
Other vertices of the paths are not adjacent to $u$ or $v$.

We shall find  appropriate $m$-sets $S \subseteq L(u)$ and $T \subseteq L(v)$, assign $S$ to $u$  and $T$  to   $v$. 
Then extend this pre-colouring of $u,v$ to an $(L, m)$-colouring of the remaining vertices of $G$, that consists of three or four paths. 
The extension to the paths are independent to each other. The difficulties lie  in  proving  the existence of such $m$-sets $S$ and $T$. 

  Assume $P$ is a path  with vertices $v_1, v_2, \ldots, v_n$ in order and  $L$ is a  $(2m+1)$-list assignment on $P$, with $v_1$ adjacent to $u$ and $v_n$ adjacent to $v$. Assume $S,T$ are the $m$-sets of colours assigned to $u,v$ respectively. 
A necessary and sufficient condition was given in \cite{4choose2} under which  $P$ has an $(L,m)$-colouring   so that $v_1$ and $v_n$  avoid the colours from $S$ and $T$.

\begin{definition} 
	\label{def-slp1}
Assume $P$ is an $n$-vertex path with vertices $v_1, v_2, \ldots, v_n$ in order.
For a list assignment $L$ of $P$, 
Let 
\begin{equation*}
\begin{array}{rl}  
X_1 =& L(v_1),   \\
X_i =& L(v_i)-X_{i-1}, i \in \{2,3,\ldots, n\},\\
S_L(P)=&\sum_{i=1}^{n}|X_i|.
\end{array}
\end{equation*}
\end{definition}

The following lemma was proved in  \cite{4choose2} (the statement is slightly different, but it does not affect the proof).

\begin{lemma}
	\label{42lemma}
	Let $P$ be an $n$-vertex path and let $L$ be a list assignment on $P$.
	If $|L(v_1)|$, $|L(v_n)| \geq m$ and $|L(v_i)|\geq 2m$ for $i \in \{2,3,\ldots, n-1\}$, then path $P$ is $(L,m)$-colourable if and only if $S_L(P) \geq nm$.
\end{lemma}

\begin{definition} 
	\label{def-slp}
	Assume $n$ is an odd integer, $P$ is an $n$-vertex path with vertices $v_1, v_2, \ldots, v_n$ in order, and $L$ is a list assignment on $P$.
	Let 
	\begin{equation*}
	\begin{array}{ll} 
	\Lambda=&\mathop{\bigcap}\limits_{x\in V(P)}L(x) ,\\
	\hat{X}_1 =& \{c \in L(v_1)-\Lambda: \mbox{the smallest index $i$ for which $c \notin L(v_i)  $ is even} \},   \\
	\hat{X}_n =& \{c \in L(v_n)-\Lambda: \mbox{the largest index $i$ for which $c \notin L(v_i)  $ is even}\}.
	\end{array}
	\end{equation*}
\end{definition}

\begin{definition}
	Assume $L$ is a $(2m+1)$-list assignment on $P$ and $S,T$ are two colour sets.
	Let  $L\ominus(S,T)$   be the list assignment obtained 
	from $L$ by deleting all colours in $S$ from $L(v_1)$, all colours in $T$ from $L(v_n)$, and leaving all other lists unchanged.
	The damage of $(S,T)$ with respect to $L$ and $P$ is   defined as 
	$$dam_{L,P}(S,T)=S_L(P)-S_{L\ominus(S,T)}(P).$$ 
\end{definition}

The following lemma  was proved in \cite{4choose2}.

\begin{lemma}[\cite{4choose2}]
	\label{slp-dam} 
	Let $L$ be a list assignment on an $n$-vertex path $P$, where $n$ is odd. For  any sets of colours $S,T$,  
	$$S_{L\ominus(S,T)}(P)=S_L(P)-(|(\Lambda\cup\hat{X}_1)\cap S|+|(\Lambda\cup\hat{X}_n)\cap T|-|\Lambda\cap S\cap T|).$$
	and 
	$$dam_{L,P}(S,T)=|\hat{X}_1\cap S|+|\hat{X}_n\cap T|+|\Lambda\cap(S\cup T)|.$$
\end{lemma}

\begin{lemma}
	\label{slp-equal}
	Let $L$ be a list assignment on an $n$-vertex path $P$ with vertices $v_1,v_2,\ldots,v_n$, where $n \geq 3$ is odd, $|L(v_1)|=l_1$ and $|L(v_i)|=l_2$ for all $i \geq 2$. Then 
	$$S_{L}(P)= l_1+\frac{n-3}{2}l_2 + \sum\limits_{\substack{k \ is \ even \\ k < n}}   |X_{k-1}-L(v_{k})|+|X_n|.$$ 
\end{lemma}

\begin{proof}
	We use induction on $n$. If $n=3$, then the lemma holds trivially. Assume   $n \geq 5$. Let $P'=P-\{v_{n-1},v_n\}$ and let $L'$ be the restriction of $L$ to $P'$, hence
	\begin{equation}
	\label{slp-proof1}
	S_{L}(P)= l_1+\frac{n-5}{2}l_2 + \sum\limits_{\substack{k \ is \ even \\ k < n-2}}|X_{k-1}-L(v_{k})|+|X_{n-2}|+|X_{n-1}|+|X_n|.
	\end{equation}
	Note that 
	\begin{align}
	\label{xn-1} 
	|X_{n-1}| & = |L(v_{n-1})-X_{n-2}| \nonumber \\ 
	& = |L(v_{n-1})|-|X_{n-2}|+|X_{n-2}- L(v_{n-1})| \nonumber \\
	& = l_2 -|X_{n-2}|+|X_{n-2}- L(v_{n-1})|.
	\end{align}
	Combining Equality (\ref{slp-proof1}) and Equality (\ref{xn-1}), we complete the proof.
\end{proof}
 
\begin{lemma}
	\label{first-lower-bound-for-slp}
	Let $L$ be a list assignment on an $n$-vertex path $P$ with vertices $v_1,v_2,\ldots,v_n$, where $n \geq 3$ is odd, $|L(v_1)|=l_1$ and $|L(v_i)|=l_2$ for all $i \geq 2$. Then
	$$S_{L}(P)\geq l_1+\frac{n-3}{2}l_2 + |\hat{X}_1|+|\hat{X}_n|+|\Lambda|.$$  
\end{lemma}

\begin{proof}
	By the definition of $\hat{X}_1$, every element of $\hat{X}_1$ appears in a set of the form $X_{k-1}-L(v_k)$ where $k$ is even. By Lemma \ref{slp-equal} and the fact that $|X_n|=|\hat{X}_n|+|\Lambda|$, the lemma holds.
\end{proof}

\begin{lemma}
	\label{second-lower-bound-for-slp}
	Let $L$ be a list assignment on an $n$-vertex path $P$ with vertices $v_1,v_2,\ldots,v_n$, where $n \geq 3$ is odd, $|L(v_1)|=l_1$ and $|L(v_i)|=l_2$ for all $i \geq 2$, then $S_L(P) \geq l_1+\frac{n-1}{2}l_2$. 
\end{lemma}

\begin{proof}
	Since $|X_1|=|L(v_1)|=l_1$ and that $|X_i|+|X_{i+1}|\geq l_2$ for $i  \geq 2$, so by the definition that $S_L(P)=\sum_{i=1}^{n}|X_i|$, the lemma holds.
\end{proof}

The following is a key lemma for the proof in this paper. It generalizes Lemma 9 in \cite{XuZhu2020}, which is a special case where $\ell$ and $k=2m-\tau$ are even.
% Also one can use the following lemma give a new proof of the well known result that every $2$-choosable graphs is $(2m,m)$-choosable, which was first proved in \cite{TuzaVoigt1996-2choosable}.}
\begin{lemma}
	\label{main-lemma}
	Let $\ell$ and $k$ be fixed  integers,  where $k \geq 1$, $\ell > k$, $0 \leq \tau \leq m$. Assume $x,y$ are non-negative integers with $x+y \le \ell$. 
	Let 
	$$F(x,y)=\sum \binom{x}{a}\binom{y}{b}\binom{\ell-x-y}{k-a-b},$$
	where  the summation is over all  non-negative integer pairs $(a,b)$ for which  $0 \leq a \leq x$, $0 \leq b \leq y$, $a+b \leq k$ and  $2a+b \geq \max\{2x+y+k+1-\ell, k+1\}$. Then 
	$$ F(x,y) \leq  \frac{1}{2}\binom{\ell}{k},$$
and the equality holds if and only if $\ell$ is even and $k$ is odd, $x=\frac{\ell}{2}$ and $y=0$.
\end{lemma}
 
Note that when $a > x$ or $b > y$, then ${x \choose a}{y \choose b} = 0$.  Also $a+b \le k$ and 
$2a+b \geq  2x+y+k+1-\ell$ implies that 
$2x+y \le \ell -k -1+ 2a+b \leq \ell+k -1$. 
Thus the summation can be restricted to $0 \le a \le x, 0 \le b \le y$, $a+b \leq k$, $2a+b \geq \max\{2x+y+k+1-\ell, k+1\}$ and $2x+y \le \ell +k-1$. The proof of Lemma \ref{main-lemma} will be given in  Section \ref{sec-main-lemma}. 

\begin{observation}
	\label{obs:about main lemma}
	If the restriction on $2a+b$ is replaced by $2a+b \geq \max\{2x+y+k+1-\ell, k+2\}$ in Lemma \ref{main-lemma}, then we have $F(x,y) < \frac{1}{2}\binom{\ell}{k}$.
\end{observation}
\begin{proof}
	It suffice to prove the observation holds when $\ell$ is even and $k$ is odd, $x=\frac{\ell}{2}$ and $y=0$. Any other case directly follows from Lemma \ref{main-lemma}. Let $H(x,0)$ be the new function which is same as $F(x,0)$ except that $2a+b \geq \max\{2x+y+k+1-\ell, k+2\} \geq k+2$.  So $H(x,0)= F(x,0)-\binom{\frac{\ell}{2}}{\frac{k+1}{2}}\binom{\frac{\ell}{2}}{\frac{k-1}{2}} < \frac{1}{2}\binom{\ell}{k}$.
\end{proof}

\section{Proof of Theorem \ref{thm-main} for $\Theta_{2r,2s,2t}$ and $\Theta_{2r+1,2s+1,2t+1}$}
\label{sec-main-thm-first}
 
Let  $G = \Theta_{2r,2s,2t}$, where $r,s,t > 1$. Let $u,v$ be the two degree $3$ vertices. Let $P^0, P^1, P^2$ be the   paths in $G-\{u, v\}$, where  $P^i = (v_1^i, v_2^i, \ldots, v_{n_i}^i)$,  $v_1^i$ is adjacent to $u$ and $v_{n_i}^i$ is adjacent to $v$. 

For the purpose of using induction, instead of proving Theorem  \ref{thm-main} directly, we shall prove a stronger result, where the list  assignment $L$ does not assign $2m+1$ colours to every vertex.  In particular, $|L(u)| = |L(v)|= \ell$, where $0 \le \ell \le 2m$.

\begin{definition}
	\label{bad simple pair}
	For a fixed indexing of $L(u)$ and $L(v)$, a {\rm couple} is a tuple of the form $(c_j,c'_j)$ for $j\in\{1,2,\ldots,\ell \}$. When we write a couple, we suppress the parentheses and simply write $c_jc'_j$. A {\rm pair} is a tuple $(S,T)$ with $S\subset L(u)$, $T\subset L(v)$, and  $|S|=|T|$. We define the \emph{size} of a pair as $|S|$. 
 	A  pair $(S,T)$ is {\rm bad with respect to $(L,P)$} if $dam_{L,P}(S,T) > S_{L}(P) - m|V(P)|$. 
	A {\rm simple pair} is a pair $(S,T)$ such that $S \cap (L(v)-T) \cap \Lambda= \emptyset$ and $T  \cap (L(u)-S) \cap \Lambda = \emptyset$.  
\end{definition}

\begin{definition}
 An indexing of colours in $ L(u)$ and $L(v)$ as $L(u)=\{c_1, c_2, \ldots, c_{\ell} \}$ and $L(v)=\{c'_1,c'_2, \ldots, c'_{\ell}\}$  is {\em consistent} if  $c_j=c'_j$ whenever $c_j \in L(u)\cap L(v)$. In other words, $\{c_i,c'_i\}\cap \{c_j,c'_j\}=\emptyset$ whenever $i\neq j$. 
\end{definition} 
 
It is easy to see that  $(S,T)$ is a simple pair if there is a consistent indexing $ L(u)=\{c_1, c_2, \ldots, c_{\ell} \}$ and $L(v)=\{c'_1,c'_2, \ldots, c'_{\ell}\}$ of colours in $L(u)$ and $L(v)$ such that    $T  = \{c'_i: c_i \in S\}$.
For convenience,  in the sequel,  we shall fix a consistent indexing of colours in $L(u)$ and $L(v)$.

\begin{observation}
	\label{obs-simple pair}
	If $S \subseteq L(u)$ and $T=\{c'_j: c_j \in S\}$, then $(S,T)$ is a simple pair (most simple pairs below are of this form).
	If $c_1c'_1$ and $c_2c'_2$ are two couples satisfying $\{c_1\} \cap \{c'_1\} \cap \Lambda =\emptyset$ and $\{c_2\} \cap \{c'_2\} \cap \Lambda =\emptyset$, then both $(c_1,c'_2)$ and $(c_2,c'_1)$ are simple pairs. If $(S_1,T_1)$ and $(S_2,T_2)$ are two simple pairs, where $S_1 \cap S_2 = \emptyset$ and $T_1 \cap T_2=\emptyset$, then $(S_1\cup S_2, T_1 \cup T_2)$ is also a simple pair. 
\end{observation}
 
The following lemma follows directly from Lemma \ref{slp-dam}.

\begin{lemma}
	\label{lem-sum}
	If  $(S_1, T_1)$ and $(S_2, T_2)$ are two pairs such that 
	$S_1 \cap (S_2 \cup T_2) = \emptyset$ and 
	$T_1 \cap (S_2 \cup T_2) = \emptyset$, then 
	$$dam_{L,P}(S_1 \cup S_2, T_1 \cup T_2) = dam_{L,P}(S_1, T_1) + dam_{L,P}(S_2, T_2).$$
	In particular,  if $(S,T)$ is a simple pair, 
	\begin{equation}
	\label{comput for dam(S,T)}
	dam_{L,P}(S,T)=\sum_{c_j\in S}dam_{L,P}(\{c_j\},\{c'_j\}).
	\end{equation}
\end{lemma}

In the following, we may write $dam_{L,P}(c,c')$ for $dam_{L,P}(\{c\},\{c'\})$.
The following observation follows from Lemma \ref{slp-dam}.

\begin{observation}
	\label{obs-couple}
	For any couple $cc'$ and $P=(v_1, v_2, \ldots, v_n)$, the following hold:
	\begin{enumerate}
		\item $dam_{L,P}(c,c') = 2$ if   $c \in \hat{X}_1 \cup \Lambda$ and $c' \in \hat{X}_n \cup \Lambda$, and moreover if $c = c'$, then $c \notin \Lambda$; 
		\item $dam_{L,P}(c,c') = 1$ if $c \in \hat{X}_1 \cup \Lambda$ or  $c' \in \hat{X}_n \cup \Lambda$ but not both  unless $c=c' \in \Lambda$;  
		\item $dam_{L,P}(c,c') = 0$ if $c \notin \hat{X}_1 \cup \Lambda$ and $c' \notin \hat{X}_n \cup \Lambda$.
	\end{enumerate}
	In particular, if $dam_{L,P}(c,c') = 2$ and $|P|=1$, then $c\neq c'$.
\end{observation}

\begin{definition}
	\label{safe-light-heavy}
	Assume $c_jc'_j$ is a couple. 
	\begin{itemize}
		\item   $c_jc'_j$ is {\rm heavy} for the internal path $P$ if $dam_{L,P}(c_j, c'_j)=2$;
		\item   $c_jc'_j$ is {\rm light} for the internal path $P$ if $dam_{L,P}(c_j, c'_j)=1$;
		\item   $c_jc'_j$ is {\rm safe} for the internal path $P$ if $dam_{L,P}(c_j, c'_j)=0$.
	\end{itemize}
\end{definition}

For each path $P^i$, let $x^{(i)},y^{(i)},z^{(i)}$ denote the number of heavy, light   and safe couples for $P^{i}$, respectively. 
Then for $i=0,1,2$, 
$$x^{(i)}+y^{(i)}+z^{(i)}=\ell, \text{ and }  dam_{L, P^i}(L(u), L(v)) = 2x^{(i)} + y^{(i)}.$$ 

Assume $m \ge \tau$ are non-negative integers and  
  $(S,T)$ is a simple pair  of size $m-\tau$. Let $a^{(i)}(S,T),b^{(i)}(S,T),c^{(i)}(S,T)$ denote the number of heavy, light and safe couples for $P^{i}$ in  $ (S,T)$, respectively.  Then for $i=0,1,2$, 
 $$a^{(i)}(S,T)+b^{(i)}(S,T)+c^{(i)}(S,T)=m-\tau \text{ and } dam_{L,P}(S,T) = 2a^{(i)}(S,T)+b^{(i)}(S,T).$$

    Let $\beta(P^{i})$ denote the   number of  bad simple pairs of size $m-\tau$  with respect to  $(L, P^i)$. We write $\hat{X}^i_1$,   $\hat{X}^i_{n_i}$ and $\Lambda^i$ for the sets $\hat{X}_1, \hat{X}_{n}, \Lambda$   calculated for $P=P^i$.

\begin{theorem}
	\label{thm:f-main first part}
	Assume $\ell$ and $\tau$ are non-negative even integers,  $L$ is a list assignment for $G$   satisfying the following:
	\begin{enumerate}[label= {(C\arabic*)}]
	 	\item \label{C1}   $\tau \leq 2\lfloor\frac{m}{2}\rfloor$ and $\ell +\tau \geq 2\lceil\frac{m}{2}\rceil$.
	 	\item \label{C2} $|L(u)|=|L(v)| = \ell$.
	 	\item \label{C3} For each $i \in \{0,1,2\}$, $|L(v^i_1)|\geq 2m-\tau$ and $|L(v^i_{n_i})| \geq  2m+1-\tau$.
	 	\item \label{C4}  $|L(w)|\geq 2m+1$ for $w \ne u,v, v^i_1, v^i_{n_i}$.
	 	\item \label{C5}  For $i=0,1,2$,
	 	$$S_L(P^i)-n_im \ge \max\{m +\frac{n_i-3}{2}+  dam_{L,P^i}(L(u),L(v))-\ell -\tau, m+\frac{n_i-1}{2} -\tau\}.$$ 
	\end{enumerate}	
	 Then there exists a set $S \subset L(u)$ and a set $T \subset L(v)$ satisfying $|S|=|T|=m-\tau$ such that for each $i$, $$dam_{L,P^i}(S,T) \leq S_L(P^i)-n_im.$$
\end{theorem}

\begin{proof} 
We prove the lemma by induction on $2\ell+\tau$.  First assume that $2\ell + \tau =2\lceil\frac{m}{2}\rceil$. Since $\ell$ and $\tau$ are non-negative, and $\ell +\tau \geq 2\lceil\frac{m}{2}\rceil$, we have  $\ell = 0$ and $\tau =2\lceil\frac{m}{2}\rceil$. Note that by \ref{C1},  $\tau \leq 2\lfloor\frac{m}{2}\rfloor$, so $m$ is even and $\tau = m$. By \ref{C5}, for each $i \in \{0,1,2\}$,  
 $S_L(P^i)-n_im \geq \frac{n_i-1}{2}+m-\tau > m-\tau=0$. Let $S=T=\emptyset$, we are done.

Thus we assume that  $2\ell + \tau > 2\lceil\frac{m}{2}\rceil$ in the sequel. Assume to the contrary, Theorem \ref{thm:f-main first part} is not true for $L$.

The following claim gives a necessary condition for a simple pair of size $m-\tau$ being bad with respect to $(L,P^i)$.  Recall that $dam_{L,P^i}(L(u),L(v))=2x^{(i)}+y^{(i)}$.   Claim \ref{claim-bad simple pair} follows from \ref{C5}  and the definition of bad pair  directly.

\begin{claim}
	\label{claim-bad simple pair}
	If $(S,T)$ is a bad simple pair of size $m-\tau$ with respect to $(L, P^i)$, then 
\begin{align*}
 	dam_{L,P^i}(S,T) &= 2a^{(i)}(S,T)+b^{(i)}(S,T) \\  
 	&\geq  \max\{2x^{(i)}+y^{(i)} +m+\frac{n_i-1}{2}-\ell-\tau, m+\frac{n_i+1}{2}-\tau\} \\
 	&\geq  \max\{2x^{(i)}+y^{(i)} +m+1-\ell-\tau, m+2-\tau\}.
\end{align*}
\end{claim}
The last inequality holds as $n_i \ge 3$.

The following claim gives an upper bound and a lower bound of the number of bad simple pairs of size $m-\tau$   with respect to $(L,P^i)$.

\begin{claim}
	\label{second case half lemma}
	For each $i \in \{0,1,2\}$, $0 < \beta(P^i)  < \frac{1}{2}\binom{\ell}{m-\tau}$.  
\end{claim}

\begin{proof}
	If a simple pair $(S,T)$ of size $m-\tau$ is bad with respect to $(L,P^i)$, then by Claim \ref{claim-bad simple pair}, $dam_{L,P^i}(S,T)  \geq \max\{2x^{(i)}+y^{(i)} +m+1-\ell-\tau, m+2-\tau\}$. Note that $a^{(i)}(S,T) + b^{(i)}(S,T) +c^{(i)}(S,T) =m-\tau$, so by Claim \ref{claim-bad simple pair} and Observation \ref{obs:about main lemma} (setting $m-\tau=k$), we have that $\beta(P^i) <  \frac{1}{2}\binom{\ell}{m-\tau}$.	
	
	If $\beta(P^i) =0$ for some $i$,   then $\beta(P^0)+\beta(P^1)+\beta(P^2)  \le  \binom{\ell}{m-\tau}-1$. So there exists a simple pair $(S,T)$ of size $m-\tau$ which is not bad with respect to any $(L,P^i)$, a contradiction to the assumption. 
\end{proof}

\begin{claim}
	\label{claim:2x+y-upper bound and xz at least one}
	For each $i \in \{0,1,2\}$, $2x^{(i)}+y^{(i)} \leq \ell+m-\tau-\frac{n_i-1}{2}$, and $x^{(i)}\geq 2$ and  $z^{(i)} \geq 1$.
\end{claim}
\begin{proof}
	If $S_{L}(P^i)- n_im \geq 2m-2\tau$, then for any simple pair $(S,T)$ of size $m-\tau$, 
	$S_L(P^i) - dam_{L,P^i}(S,T) \ge n_im$ (as  $dam_{L,P^{i}}(S,T) \leq 2m-2\tau$), hence $(S,T)$ is not bad with respect to $(L,P^i)$, which means that $\beta(P^i)=0$, a contradiction to Claim \ref{second case half lemma}.  Thus we may assume that  $S_{L}(P^i) - n_im\leq 2m-2\tau-1$. It follows from \ref{C5} that	
	\begin{equation*}
	\label{2x+y-upperbound-6m-1}
	2x^{(i)}+y^{(i)} =dam_{L,P^i} (L(u),L(v)) \leq 2m-2\tau -1-m -\frac {n_i-3}{2} +\ell+\tau = \ell+m-\tau-\frac{n_i-1}{2}.
	\end{equation*}
	Thus we proved the first part.

	Assume $x^{(i)} \leq 1$ for some $i \in \{0,1,2\}$,   then   for every simple pair $(S,T)$ of size $m-\tau$, $dam_{L,P^i}(S,T) \leq 2\times 1 +(m-\tau-1)= m-\tau+1$, contrary to Claim \ref{claim-bad simple pair}.  Thus $x^{(i)} \geq 2$.
	
	Assume $z^{(i)}=0$, then $x^{(i)}+y^{(i)}= \ell$ and for any simple pair $(S,T)$ of size $m-\tau$, $a^{(i)}(S,T)+b^{(i)}(S,T) = m-\tau$. By Claim \ref{claim-bad simple pair}, we have 
	\begin{align*} 
	a^{(i)}(S,T) +m-\tau & = 2a^{(i)}(S,T)+b^{(i)}(S,T)  \\
	& \geq  2x^{(i)}+y^{(i)} +m+1-\ell-\tau\\
	& = x^{(i)}+1+m-\tau. 
	\end{align*}
	This implies that $a^{(i)}(S,T) \geq x^{(i)}+1$, in contrary to the fact that  $a^{(i)}(S,T) \leq x^{(i)}$. 
\end{proof}

\begin{claim}
	\label{claim:l+t at least m+1 and m-t is at least 2}
	$\ell+\tau \geq m+1$ and $\tau \leq 2\lfloor\frac{m}{2}\rfloor-2 \le m-2$. 
\end{claim}
\begin{proof}
Suppose to the contrary, $\ell+\tau<m+1$. By \ref{C1}, we have $\ell+\tau=m$. Let $S=L(u)$, $T=L(v)$.  By \ref{C5}, for $i = 0,1,2$, 
$$dam_{L,P^i}(S,T) =2x^{(i)}+y^{(i)} \leq  S_L(P^i)-n_im,$$
contrary to the assumption. This proves the first inequality.
 	
Assume to the contrary that $\tau > 2\lfloor\frac{m}{2}\rfloor - 2$. As $\tau$ is even and  $\tau \le 2\lfloor\frac{m}{2}\rfloor$, we have 
$\tau = 2\lfloor\frac{m}{2}\rfloor$. 
If $m$ is even,   then $\tau=m$ and we take $S=T=\emptyset$.  By \ref{C5}, $dam_{L,P^i}(S,T) = 0 < S_L(P^i) -n_im$ for $i=0,1,2$, a contradiction. 

Assume $m$ is odd, then we have $\tau=m-1$. By \ref{C5}, $S_L(P^i) -n_im \geq m+ 1  -\tau  \geq  2$. Let $S=\{c\}$ and $T=\{c'\}$ for any couple $cc'$. Then $dam_{L,P^i}(S,T) \le 2 \le S_L(P^i) -n_im$ for $i=0,1,2$, a contradiction. 
\end{proof}

\begin{claim}
	\label{all at least light}
	 There does not exist a simple pair $(D_u,D_v)$ such that $|D_u|=|D_v|=d \le \ell -m+\tau$ is even,  and $dam_{L,P^i}(D_u,D_v) \geq d$ for each  $i \in \{0,1,2\}$. 
\end{claim}

\begin{proof}
	Assume $(D_u,D_v)$ is such a simple pair. Let $L'$ be a new list assignment for $G$ with $L'(u)= L(u)-D_u$, $L'(v)=L(v)-D_v$, $L'(w)=L(w)$ for $w \in V(G) \setminus \{u,v\}$. 
	
	\ref{C1}-\ref{C4} of Theorem \ref{thm:f-main first part} are easily seen to be satisfied by $L'$, with $ \ell'= \ell - d$ and $\tau'=\tau$.
	 
 As $L'(w)=L(w)$ for $w\in V(G) \setminus\{u,v\}$, so for each $i\in \{0,1,2\}$, $dam_{L',P^i}(L'(u),L'(v))=dam_{L,P^i}(L'(u),L'(v))$ and $S_L(P^i) = S_{L'}(P^i)$. Therefore, by  Lemma \ref{lem-sum}, $$dam_{L,P^i}(L(u), L(v)) = dam_{L',P^i}(L'(u), L'(v)) + dam_{L,P^i}(D_u,D_v) \geq dam_{L',P^i}(L'(u), L'(v)) +d,$$ 
  and 
	\begin{align*} 
S_{L'}(P^i) - n_im   & =   S_L(P^i)-n_im \\
					 & \geq  \max \{ m+ \frac{n_i-3}{2}+dam_{L,P^i}(L(u), L(v))-\ell-\tau, m+\frac{n_i-1}{2}-\tau\} \\
				  & \geq  \max \{ m+ \frac{n_i-3}{2}+dam_{L',P^i}(L'(u), L'(v))-\ell'-\tau,
					 m+\frac{n_i-1}{2}-\tau\}.
	\end{align*}
I.e.,	\ref{C5} is also satisfied by $L'$.
	By induction hypothesis,   there exists a pair $(S,T)$, with $|S|=|T|=m-\tau$, $S \subseteq L'(u) \subseteq L(u), T \subseteq L'(v) \subseteq L(v)$,  such that for each $i \in \{0,1,2\}$, $dam_{L,P^i}(S,T) =dam_{L',P^i}(S,T) \leq S_L(P^i)-n_im$.  
	
	This completes the proof of this  claim.
\end{proof}

\begin{claim}
	\label{all at most light}
	There does not exist a simple pair $(D_u,D_v)$ such that $0 < |D_u|=|D_v|=d \le m-\tau$ is even,  and $dam_{L,P^i}(D_u,D_v) \leq d$ for each  $i \in \{0,1,2\}$. 
\end{claim} 

\begin{proof}
	Assume   $(D_u,D_v)$ is such a simple pair. Let $L'$ be a new list assignment for $G$ with $L'(u)=L(u)-D_u$, $L'(v)=L(v)-D_v$, for each $i$,  $L'(v^i_1)= L(v^i_1)-D_u$, $L'(v^i_{n_i})=L(v^i_{n_i})-D_v$,  $L'(v^i_j)=L(v^i_j)$ where $1 < j < n_i$.
	
	Observe that \ref{C1}-\ref{C4} of Theorem \ref{thm:f-main first part} are  satisfied by $L'$, with $ \ell'= \ell - d$ and $\tau'=\tau+d$.  Note that
	$  S_{L'}(P^i)= S_{L}(P^i) - dam_{L,P^i}(D_u,D_v) \ge S_{L}(P^i) -d$.
	As $S_L(P^i)-n_im \ge m+ \frac{n_i-1}{2}-\tau$, we have  $S_{L'}(P^i)-n_im \geq m+\frac{n_i-1}{2}-\tau'$.  On the other hand,  as $dam_{L',P^i}(L'(u),L'(v))=dam_{L,P^i}(L'(u),L'(v))$,  by Lemma \ref{lem-sum}, $dam_{L',P^i}(L'(u),L'(v)) = dam_{L,P^i}(L(u), L(v)) - dam_{L,P^i}(D_u,D_v)$. So 
		\begin{align*}
			S_{L'}(P^i) - n_im  & =  S_L(P^i)-n_im -dam_{L,P^i}(D_u,D_v) \\
			& \geq   m+ \frac{n_i-3}{2}+dam_{L,P^i}(L(u), L(v))
		-dam_{L,P^i}(D_u,D_v)-\ell-\tau,   \\
			& =   m+ \frac{n_i-3}{2}+dam_{L',P^i}(L'(u), L'(v))-\ell'-\tau.
		\end{align*} 
	Therefore,	\ref{C5} is also satisfied by $L'$.

	By induction,   there exists a pair $(S',T')$, where $|S'|=|T'|=m-\tau' = m-\tau-d$ such that for every $i$, $$dam_{L',P^i}(S',T') \leq S_{L'}(P^i)-n_im.$$  Let $S=S' \cup D_u$ and  $T=T' \cup D_v$.   As $S' \cap D_u = \emptyset$ and $T' \cap D_v = \emptyset$,  $dam_{L',P^i}(S',T'))=dam_{L,P^i}(S',T')$.  Thus we have $|S|=|T|=m-\tau$ and  
	\begin{align*} 
			dam_{L,P^i}(S,T) & \leq  dam_{L,P^i}(D_u,D_v)+dam_{L,P^i}(S',T') \\
							 & \leq dam_{L,P^i}(D_u,D_v)+  S_{L'}(P^i)-n_im\\
							 & = S_{L}(P^i)-n_im.
	\end{align*} 
	 
	This completes the proof of Claim \ref{all at most light}.
\end{proof}

Observe that it follows from Claim \ref{claim:l+t at least m+1 and m-t is at least 2} that $m-\tau \ge 2$. So there does not exist $(D_u, D_v)$ such that $|D_u|=|D_v|=2$ and $dam_{L,P^i}(D_u,D_v) \leq 2$ for each $i \in \{0,1,2\}$.

\begin{claim}
	\label{claim:if l+t =m+1 no all at least light}
	If $\ell+\tau=m+1$, then there does not exist a pair $(c,c')$   (not necessarily a simple pair)   such that  $dam_{L,P^i}(L(u)\setminus c,L(v)\setminus c') \leq dam_{L,P^i}(L(u),L(v))-1$ for $i=0,1,2$. Consequently, we have the following.
	\begin{enumerate}[label= {(\arabic*)}]
		\item \label{claim:case-at least light couple} There is no couple $(c,c')$ satisfying $dam_{L,P^i}(c,c') \geq 1$ for all $i$.
		\item \label{claim:l+t=m+1 heavy for at most one}  Every couple is heavy for at most one internal path.
	\end{enumerate}
\end{claim}
\begin{proof}
	If this is not true, then let $S=L(u)\setminus c$ and $T=L(v)\setminus c'$ (and hence $m-\tau = |S|=\ell-1$). By \ref{C5}  and the assumption, for $i \in \{0,1,2\}$,
	\begin{align*} 
	S_{L}(P^i) - n_im & \geq m+ \frac{n_i-3}{2} +2x^{(i)}+y^{(i)}-\ell-\tau \\
	& \geq 2x^{(i)}+y^{(i)} -1 \\
	&  =dam_{L,P^i}(L(u),L(v))-1 \\
	& \geq dam_{L,P^i}(S,T).   
	\end{align*}  
	Hence, $(S,T)$ is a pair   satisfying Theorem \ref{thm:f-main first part}, a contradiction.	
	
	  \ref{claim:case-at least light couple} follows from Equality (\ref{comput for dam(S,T)}), as  if $(c,c')$ is a couple, i.e., a simple pair of size $1$,  then $dam_{L,P^i}(L(u),L(v)) = dam_{L,P^i}(L(u)\setminus c,L(v)\setminus c') + dam_{L,P^i}(c,c')$.
	  
	 For \ref{claim:l+t=m+1 heavy for at most one},  suppose to the contrary, $c_0c'_0$ is heavy for $P^0$ and $P^1$. By \ref{claim:case-at least light couple}, $c_0c'_0$ is safe for $P^2$.	As $x^{(2)} \geq 2$, there exists a couple $c_1c'_1$ which is heavy for $P^2$. We claim that for $i=0,1,2$, 
	 	\begin{equation}
	 	\label{ineq: difference is at least one}
	 	dam_{L,P^i}(\{c_0,c_1\},\{c'_0,c'_1\})-dam_{L,P^i}(c_0,c'_1)\geq 1.
	 	\end{equation} 
	 If $i=2$, then $dam_{L,P^i}(\{c_0,c_1\},\{c'_0,c'_1\}=2$ and  $dam_{L,P^i}(c_0,c'_1)=1$,  which implies that Inequality (\ref{ineq: difference is at least one}) holds.
	 For $i=0$ or $1$, if $c_1c'_1$ is not safe for $P^i$, then $dam_{L,P^i}(\{c_0,c_1\},\{c'_0,c'_1\} \geq 3$ and $dam_{L,P^i}(c_0,c'_1) \leq 2$, hence Inequality (\ref{ineq: difference is at least one})   holds. 
	 If $c_1c'_1$ is   safe for $P^i$, then $dam_{L,P^i}(\{c_0,c_1\},\{c'_0,c'_1\} =2$ and $dam_{L,P^i}(c_0,c'_1) =1$, so Inequality (\ref{ineq: difference is at least one}) also holds. 
	 	
 Let $S=L(u)\setminus c_1$, and $T=L(v) \setminus c'_0$. By Lemma \ref{lem-sum} and Inequality (\ref{ineq: difference is at least one}), for $i=0,1,2$,
	 \begin{align*}
	 	dam_{L,P^i}(S,T)  &=  dam_{L,P^i}(L(u)\setminus\{c_0,c_1\},L(v)\setminus\{c'_0,c'_1\})+dam_{L,P^i}(c_0,c'_1) \\
	 	& =   dam_{L,P^i}(L(u), L(v)) -dam_{L,P^i}(\{c_0,c_1\},\{c'_0,c'_1\}) +dam_{L,P^i}(c_0,c'_1) \\
	 	&\leq  dam_{L,P^i}(L(u), L(v))-1,
	 \end{align*}
   a  contradiction. 
\end{proof}
 
	\begin{claim} 
		\label{claim:evey couple is HSL} 
		The following hold:	
		\begin{enumerate}[label= {(\arabic*)}]
			\item \label{claim:case-not safe and heavy for two} Every couple is safe (respectively, heavy) for at most one internal path. 
			\item \label{claim:case-l+t>m+1} If $\ell+\tau \geq m+2$, then no couple   is light for exactly two internal paths.
			Moreover, there is at most one couple which is light for all internal paths.
			\item \label{claim:case-l+t=m+1 not light for two} If $\ell+\tau \leq m+1$, then every couple is light for at most one internal path.   
		\end{enumerate}
		
	\end{claim}

\begin{proof}  
	\ref{claim:case-not safe and heavy for two}. Assume to the contrary,   $c_jc'_j$ is safe for two paths,  say for both $P^0$ and $P^1$. If $c_jc'_j$ is also safe for $P^2$, then for any other couple $c_kc'_k$, we know that $(\{c_j,c_k\},\{c'_j,c'_k\})$ is a simple pair of size $2$ contradicting  Claim \ref{all at most light} (Recall that $m-\tau \geq 2$ by Claim \ref{claim:l+t at least m+1 and m-t is at least 2}).  Thus $c_jc'_j$ is not safe for $P^2$.   As $z^{(2)} \geq 1$, there exists a couple $c_kc'_k$ which is safe for $P^2$. It follows that $(\{c_j,c_k\},\{c'_j,c'_k\})$ is a simple pair of size $2$ contradicting  Claim \ref{all at most light}. 
	
	Next, we shall prove that every couple is heavy for at most one path.  It is true if $\ell+\tau=m+1$ by Claim \ref{claim:if l+t =m+1 no all at least light}\ref{claim:l+t=m+1 heavy for at most one}.  
	Assume $\ell +\tau \geq m+2$. 
	Assume there is a  couple   $c_jc'_j$
	which is heavy for at least two internal paths, say $P^0$ and $P^1$. If  $c_jc'_j$ is also heavy for $P^2$, then for any other couple $c_kc'_k$,   $(\{c_j,c_k\},\{c'_j,c'_k\})$ is a simple pair of size $2$ contradicting  Claim \ref{all at least light} (we need the assumption $\ell +\tau \geq m+2$ so that we can use Claim \ref{all at least light} with $d=2$). So $c_jc'_j$ is not heavy for $P^2$. As $x^{(2)} \ge 2$, there exists a couple $c_kc'_k$ which is heavy for $P^2$, Then  $(\{c_j,c_k\},\{c'_j,c'_k\})$ is also a simple pair of size $2$ contradicting  Claim  \ref{all at least light}.

	\ref{claim:case-l+t>m+1}. 
	Assume to the contrary that there is a couple $c_jc'_j$ which is light for exactly two internal paths, say $P^0$ and $P^1$, and  $c_jc'_j$ is either heavy or safe for $P^2$. 
	
	First assume that  $c_jc'_j$ is heavy for $P^2$. Note that by Claim \ref{claim:2x+y-upper bound and xz at least one}, $z^{(2)} \geq 1$, then there exists a distinct couple $c_kc'_k$ which is safe for $P^2$. By  the fact that no couple is  safe    for two internal paths (by \ref{claim:case-not safe and heavy for two} of this claim), $c_kc'_k$ is safe for neither $P^0$ nor $P^1$. Then  $(\{c_j,c_k\},\{c'_j,c'_k\})$ is a simple pair of size $2$ contradicting   Claim \ref{all at least light}. 
	
	So $c_jc'_j$ is safe for $P^2$. As $x^{(2)} \geq 2$ (by Claim \ref{claim:2x+y-upper bound and xz at least one}), there exists a distinct couple $c_kc'_k$ which is heavy for $P^2$. By  the fact that no couple is heavy for two internal paths,  $c_kc'_k$ is heavy for neither $P^0$ nor $P^1$. Then  $(\{c_j,c_k\},\{c'_j,c'_k\})$ is a simple pair  of size $2$  contradicting   Claim \ref{all at most light}.  
	
	For the ``moreover'' part, if there are  two couples which are light for all internal paths, then two such couples comprise a simple pair of size $2$ which contradicts Claim \ref{all at most light}.  
	This completes the proof of   \ref{claim:case-l+t>m+1}.

 	\ref{claim:case-l+t=m+1 not light for two}. Assume to the contrary, $c_ic'_i$ is light for at two paths, say $P^0$ and $P^1$. It follows from Claim \ref{claim:if l+t =m+1 no all at least light}\ref{claim:case-at least light couple} that $c_ic'_i$ is safe for $P^2$. By Claim \ref{claim:2x+y-upper bound and xz at least one}, $x^{(2)} \geq 1$, so assume that $c_jc'_j$ is heavy for $P^2$. By \ref{claim:case-not safe and heavy for two} of this claim, $c_jc'_j$ is not heavy for both $P^0$ and $P^1$, so $(\{c_i,c_j\},\{c'_i,c'_j\})$ is a simple pair of size $2$ contradicting  Claim \ref{all at most light} 
 		
	This completes the proof of Claim \ref{claim:evey couple is HSL}. 
\end{proof}

	\medskip
	
Since $x^{(i)} \ge 2$ for $i=0,1,2$ (by Claim \ref{claim:2x+y-upper bound and xz at least one}) and no couple is heavy for two internal paths (by  Claim \ref{claim:evey couple is HSL}\ref{claim:case-not safe and heavy for two} ), there exist 
distinct couples $c_ic'_i$ for $i=0,1,\ldots, 5$ such that 
 \begin{itemize}
 	\item $c_0c'_0$ and $c_1c'_1$ are   heavy for $P^0$.
 	\item $c_2c'_2$ and $c_3c'_3$ are   heavy for $P^1$.
 	\item $c_4c'_4$ and $c_5c'_5$ are heavy for $P^2$.
 \end{itemize}  
	Without loss of generality, we may   assume that $c_0c'_0$ is  light for $P^1$ and safe for $P^2$  (by Claim \ref{claim:evey couple is HSL}\ref{claim:case-l+t>m+1} and Claim \ref{claim:evey couple is HSL}\ref{claim:case-l+t=m+1 not light for two},   $c_0c'_0$ cannot be light for both $P^1$ and $P^2$, and by Claim \ref{claim:evey couple is HSL}\ref{claim:case-not safe and heavy for two},   $c_0c'_0$ cannot be safe for both $P^1$ and $P^2$). Then both $c_4c'_4$ and $c_5c'_5$ are light for $P^0$, for otherwise, 
	$(\{c_0,c_4\},\{c'_0,c'_4\})$ or $(\{c_0,c_5\},\{c'_0,c'_5\})$  is a  simple pair of size $2$ which contradicts Claim \ref{all at most light}. 
	 Consequently, by Claim \ref{claim:evey couple is HSL}, both $c_4c'_4$ and $c_5c'_5$ are safe for $P^1$.
	 
	Similarly,  both $c_2c'_2$ and $c_3c'_3$ are light for $P^2$, and  safe for $P^0$, since otherwise, $(\{c_2,c_4\},\{c'_2,c'_4\})$ or  $(\{c_3,c_4\},\{c'_3,c'_4\})$ is a  simple pair of size $2$ which contradicts  Claim \ref{all at most light}.

	Also $c_1c'_1$ is light for $P^1$, safe for $P^2$, for otherwise, $(\{c_1,c_2\},\{c'_1,c'_2\})$ is a  simple pair of size $2$ which contradicts  Claim \ref{all at most light}. See Table \ref{table1}.
 
 \begin{table}[H]
	\centering
	\begin{tabular}{c|c|c|c|c|c|c|c} 
		\bottomrule[1pt] \hline $L(u)$ & $c_0$ & $c_1$ & $c_2$ & $c_3$ & $c_4$ & $c_5$ & $\cdots$   \\	
		\hline $P^0$  & heavy & heavy     & safe     & safe     & light    & light     &   $\cdots$      \\
		\hline $P^1$  & light & light     & heavy     & heavy    & safe    & safe     &  $\cdots$      \\
		\hline $P^2$  & safe & safe     & light     & light     & heavy    & heavy     &    $\cdots$    \\ 	 
		\hline $L(v$) & $c'_0$  & $c'_1$ & $c'_2$ & $c'_3$ & $c'_4$ & $c'_5$ &  $\cdots$    \\	
	 	\toprule[1.5pt]
	\end{tabular} 
	\caption{$dam_{L,P^i}(c_i,c'_i)$}
	 \label{table1}
\end{table}

	If $\tau \leq 2\lfloor\frac{m}{2}\rfloor-6$, then  $(\{c_0,c_1,\ldots,c_5\},\{c'_0,c'_1,\ldots,c'_5\})$ is a simple pair of size $6$ which contradicts  Claim \ref{all at most light}.

	Assume $\tau =2\lfloor\frac{m}{2}\rfloor-4$.  If $m$ is even, then $m-\tau=4$.
	By \ref{C5}, $	S_L(P^i)-n_im \geq \frac{n_i-1}{2}+m -\tau \geq m-\tau+1  = 5$. Let $S=\{c_0,c_1,c_2,c_4\} $ and $T=\{c'_0,c'_1,c'_2,c'_4\}$. Then $dam_{L,P^i}(S,T) \leq 5$, we are done. 
	If $m$ is odd, then $m-\tau=5$ and $S_L(P^i)-n_im \geq 6$. Let  $S=\{c_0,c_1,c_2,c_3,c_4\} $ and $T=\{c'_0,c'_1,c'_2,c'_3,c'_4\}$, and we have  $dam_{L,P^i}(S,T) \leq 6$, we are also done.
	
 	Assume $\tau =2\lfloor\frac{m}{2}\rfloor-2$.  If $m$ is even, then $m-\tau=2$, and  $S_L(P^i)-n_im \geq 3$. Let $S=\{c_0,c_2\} $ and $T=\{c'_0,c'_2\}$. Then $dam_{L,P^i}(S,T) \leq 3$ for each $i$, so we are done. If $m$ is odd, $m-\tau=3$, and $S_L(P^i)-n_im \geq 4$. Let  $S=\{c_0,c_2,c_4\} $ and $T=\{c'_0,c'_2,c'_4\}$. Then  $dam_{L,P^i}(S,T) \leq 3$ for each $i$, and we are also done.
	
	This completes the proof of Theorem \ref{thm:f-main first part}.
	\end{proof}
	
	 \begin{corollary}
		\label{coro-theta-even}
		Suppose $G=\Theta_{2r,2s,2t}$ with $r,s,t \geq 2$, $u,v$ are the two vertices of degree $3$, L is a list assignment of G with $|L(x)|=2m$ if  $x \in \{u,v\}\cup N_G(u)$ and $|L(x)| \geq 2m+1$ for the other vertices. Then $G$ is $(L,m)$-colourable.
	\end{corollary}
	\begin{proof}
	Let $\ell =2m$ and $\tau=0$. By Lemma \ref{slp-dam}, $|\hat{X}^i_1|+|\hat{X}^i_n|+|\Lambda^i| \geq |\hat{X}^i_1 \cap L(u)|+|\hat{X}^i_n\cap L(v)|+|\Lambda^i\cap (L(u)\cup L(v))|=dam_{L,P^i}(L(u),L(v))$. Setting $l_1=2m$, $l_2=2m+1$, by Lemma \ref{first-lower-bound-for-slp},  $S_{L}(P) -n_im \ge   \frac{n_i-3}{2}-m+    |\hat{X}^i_1|+|\hat{X}^i_n|+|\Lambda^i| \geq \frac{n_i-3}{2}+m-\ell-\tau +  dam_{L,P^i}(L(u),L(v))$. On the other hand, by Lemma \ref{second-lower-bound-for-slp}, $S_{L}(P^i) \geq l_1+\frac{n_i-1}{2}l_2=n_im+m+\frac{n_i-1}{2}-\tau$. So \ref{C5} holds. Observe that $L$, $\ell$, $\tau$ also satisfies \ref{C1}-\ref{C4}. 
	By Theorem \ref{thm:f-main first part},
	 there exist     $S \subset L(u)$, $T \subset L(v)$ such that $|S|=|T|=m$ and $dam_{L,P^i}(S,T) \leq S_L(P^i) -n_im$,  which implies that $G$ is $(L,m)$-colourable.
	\end{proof}  
	
	\begin{corollary}
		\label{coro-theta-odd}
		Suppose $G=\Theta_{2r+1,2s+1,2t+1}$ with $r,s,t \geq 1$, $u,v$ are the two vertices of degree $3$, $L$ is a list assignment of G with $|L(x)|=2m$ if $x \in \{u,v\}$ and $|L(x)| \geq 2m+1$ for the other  vertices. Then $G$ is $(L,m)$-colourable.
	\end{corollary}
	\begin{proof}
		Let $G'=\Theta_{2r+2,2s+2,2t+2}$ be obtained from $G$ by splitting $u$ into three vertices $u_1, u_2, u_3$ of degree $1$ (each adjacent to one neighbor of $u$), adding a vertex $u'$ adjacent to $u_1,u_2,u_3$.   Let $L'$ be a list assignment of $G'$ with $L'(x)=L(u)$ if $x \in \{u', u_1,u_2,u_3\}$,  and $L'(x)=L(x)$ for other vertices.	
		By Corollary \ref{coro-theta-even}, $G'$ is $(L',m)$-colourable and assume $\phi'$ is such an $(L',m)$-colouring of $G'$. Observe that for each $x \in \{u_1,u_2,u_3\}$, $\phi'(x)=L'(u')-\phi'(u)$. Now let $\phi$ be a $(L,m)$-colouring of $G$ as follows: $\phi(u)=\phi'(u_1)$, and $\phi(x)=\phi'(x)$ for $x \in V(G)-\{u\}$. It is clear that $\phi$ is a proper $(L,m)$-colouring of $G$.
	\end{proof}

\section{Proof of Theorem \ref{thm-main} for  $\Theta_{2,2,2,2p}$}
\label{sec-main-thm-second}

In this section,  $G=\Theta_{2,2,2,2p}$ with $p \geq 1$, $u,v$ are the two vertices of degree $4$, and $P^0,P^1,P^2, P^3$ are the four paths of $G-\{u,v\}$.
Similarly, assume $P^i=(v^i_1,v^i_2,\ldots,v^i_{n_i} )$,  $v^i_1$ is adjacent to $u$ and $v^i_{n_i}$ is adjacent to $v$, where $n_0=n_1=n_2=1$ and $n_3 \geq 1$.  
We shall use the notation introduced in Section \ref{sec-main-thm-first}.
  
Similarly, instead of proving directly that $G$ is $(2m+1,m)$-choosable, we prove the following stronger and more technical result.

\begin{theorem}
	\label{thm:s-main-second-part-stronger}
	Assume $\ell$ and $\tau$ are non-negative  integer, $L$ is a list assignment for $G$ satisfying the following:
	\begin{enumerate}[label= {(T\arabic*)}]
		\item \label{T1}   $\tau \leq m$ and $\ell +\tau \geq m$.
		\item \label{T2} $|L(u)|=|L(v)| = \ell \ge 0$.
		\item \label{T3} For each $i \in \{0,1,2,3\}$,  $|L(v^i_1)| \geq 2m+1-\tau$. If $n_3 \geq 3$, then $|L(v^3_{n_3})|\geq 2m+1-\tau$. 
		\item \label{T4}  $|L(w)|\geq 2m+1$ for $w \ne u,v, v^i_1, v^i_{n_i}$.
		\item \label{T5}  For $i=0,1,2,3$,    $$S_L(P^i)-n_im \ge \max\{ \frac{n_i+1}{2}+ m -\ell-\tau + dam_{L,P^i}(L(u),L(v)), \frac{n_i+1}{2}+m-\tau\}.$$
	\end{enumerate}	
	Then there exists a set $S \subset L(u)$ and a set $T \subset L(v)$ satisfying $|S|=|T|=m-\tau$ such that for each $i$, $$dam_{L,P^i}(S,T) \leq S_L(P^i)-n_im.$$
\end{theorem}

\begin{proof}  Before the proof, we observe that Theorem \ref{thm:s-main-second-part-stronger} is similar to Theorem \ref{thm:f-main first part}. However, besides these two theorems refer to different graphs, there is another subtle difference: $\ell$ and $\tau$ are allowed to be odd in Theorem \ref{thm:s-main-second-part-stronger}.

The proof is by induction on $2\ell+\tau$.  First assume that $2\ell + \tau =m$. Since $\ell + \tau \geq m$ and $\ell, \tau$ are non-negative, we have that $\ell = 0$ and $\tau = m                                                                                                                                                                                                                                                                                                                                                                                                                                                                                                                                                                                                                                                                                                                                                                                                                                                                                                                                                                                                                                                                                                                                                                                                                                                                                                                                                                                                                                                                                                                                                                                                                                                                                                                                                                                                                                                                                                                                                                                                                                                                                                                                                                                                                                                                                                                                                                                                                                                                                                                                                                                                                                                                                                                                                                                                                                                                                                                                                                                                                                                                                                                                                                                                                                                                                                                                                                                                                                                                                                                                                                                                                                                                                                                                                                                                                                                                                                                                                                                                                                                                                                                                                                                                                                                                                                                                                                                                                                                                                                                                                                                                                                                                                                                                                                                                                                                                                                                                                                                                                                                                                                                                                                                                                                                                                                                                                                                                                                                                                                                                                                                                                                                                                                                                                                                                                                                                                                                                                                                                                                                                                                                                                                                                                                                                                                                                                                                                                                                                                                                                                                                                                                                                                                                                                                                                                                                                                                                                                                                                                                                                                                                                                                                                                                                                                                                                                                                                                                                                                                                                                                                                                                                                                                                                                                                                                                                                                                                                                                                                                                                                                                                                                                                                                                                                                                                                                                                                                                                                                                                                                                                                                                                                                                                                                                                                                                                                                                                                                                                                                                                                                                                                                                                                                                                                                                                                                                                                                                                                                                                                                                                                                                                                                                                                                                                                                                                                                                                                                                                                                                                                                                                                                                                                                                                                                                                                                                                                                                                                                                                                                                                                                                                                                                                                                                                                                                                                                                                                                                                                                                                                                                                                                                                                                                                                                                                                                                                                                                                                                                                                                                                                                                                                                                                                                                                                                                                                                                                                                                                                                                                                                                                                                                                                                                                                                                                                                                                                                                                                                                                                                                                                                                                                                                                                                                                                                                                                                                                                                                                                                                                                                                                                                                                                                                                                                                                                                                                                                                                                                                                                                                                                                                                                                                                                                                                                                                                                                                                                                                                                                                                                                                                                                                                                                                                                                                                                                                                                                                                                                                                                                                                                                                                                                                                                                                                                                                                                                                                                                                                                                                                                                                                                                                                                                                                                                                                                                                                                                                                                                                                                                                                                                                                                                                                                                                                                                                                                                                                                                                                                                                                                                                                                                                                                                                                                                                                                                                                                                                                                                                                                                                                                                                                                                                                                                                                                                                                                                                                                                                                                                                                                                                                                                                                                                                                                                                                                                                                                                                                                                                                                                                                                                                                                                                                                                                                                                                                                                                                                                                                                                                                                                                                                                                                                                                                                                                                                                                                                                                                                                                                                                                                                                                                                                                                                                                                                                                                                                                                                                                                                                                                                                                                                                                                                                                                                                                                                                                                                                                                                                                                                                                                                                                                                                                                                                                                                                                                                                                                                                                                                                                                                                                                                                                                                                                                                                                                                                                                                                                                                                                                                                                                                                                                                                                                                                                                                                                                                                                                                                                                                                                                                                                                                                                                                                                                                                                                                                                                                                                                                                                                                                                                                                                                                             $. By \ref{T5}, for each $i \in \{0,1,2,3\}$, 
$S_L(P^i)-n_im \geq \frac{n_i+1}{2}+m-\tau \geq 1$.
Let $S=L(u)= \emptyset$, $T=L(v) =\emptyset$, and we are done.  

Assume that  $2\ell+\tau \geq m+1$. If $\ell+\tau=m$, then let $S=L(u)$, $T=L(v)$.  If $\ell + \tau = m+1$, then we let $(S,T)$ be arbitrary simple pair of size $(\ell-1)$. 
In either case, for $i=0,1,2,3$,
\begin{align*} 
dam_{L,P^i}(S,T)  & \leq  dam_{L,P^i}(L(u),L(v))  \\
				  & \leq S_L(P^i)-n_im-\frac{n_i+1}{2}-m+\ell+\tau \\
				  & \leq  S_L(P^i)-n_im-\frac{n_i+1}{2}-m+(m+1) \\
				  & = S_L(P^i)-n_im-\frac{n_i-1}{2} \\
				  & \leq  S_L(P^i)-n_im.
\end{align*}

So we are done. Thus we assume that $\ell +\tau \geq m+2$. 

If $\tau =m$, then let $S=T=\emptyset$ and we are done. If $\tau=m-1$, then let $(S,T)$ be any simple pair of size $1$, we have $dam_{L,P^i}(S,T) \leq 2 \le S_L(P^i)-n_im$ by \ref{T5}.   

In the sequel, we assume $\tau \leq m-2$. Assume to the contrary that Theorem \ref{thm:s-main-second-part-stronger} is not true for $L$.

\begin{claim}
	\label{claim:s-all at least light}
There is no simple pair $(D_u,D_v)$ such that $|D_u|=|D_v|=d \le \ell -m+\tau$, and for each $i \in \{0,1,2,3\}$,   $x^{(i)}=0$ or  $dam_{L,P^i}(D_u,D_v) \geq d$. 
\end{claim}

\begin{proof}
	Assume $(D_u,D_v)$ is such a pair. Let $L'$ be a new list assignment for $G$ with $L'(u)= L(u)-D_u$, $L'(v)=L(v)-D_v$, $L'(w)=L(w)$ for $w \in V(G) \setminus \{u,v\}$. 
	
	\ref{T1}-\ref{T4} of Theorem \ref{thm:s-main-second-part-stronger}  are easily seen to be satisfied by $L'$, with $ \ell'= \ell - d$ and $\tau'=\tau$. Note that $S_{L'}(P^i)-n_im=S_{L}(P^i)-n_im \geq m+\frac{n_i+1}{2}-\tau=m+\frac{n_i+1}{2}-\tau'$. On the other hand,   note that $dam_{L',P^i}(L'(u),L'(v))=dam_{L,P^i}(L'(u),L'(v))$.  So if $dam_{L,P^i}(D_u,D_v) \geq d$, then by Lemma \ref{lem-sum},   $dam_{L,P^i}(L(u), L(v)) = dam_{L',P^i}(L'(u), L'(v)) + dam_{L,P^i}(D_u,D_v) \geq dam_{L',P^i}(L'(u), L'(v)) +d$.  
	So
	\begin{eqnarray*}
		S_{L'}(P^i) - n_im  & = & S_L(P^i)-n_im \\
		& \geq &  m+ \frac{n_i+1}{2}+dam_{L,P^i}(L(u), L(v))-\ell-\tau,   \\
		& \geq &  m+ \frac{n_i+1}{2}+dam_{L',P^i}(L'(u), L'(v))-\ell'-\tau.
	\end{eqnarray*}
	If $x^{(i)}=0$, then  for every couple $cc'$, $dam_{L,P^i}(c, c') \leq 1$,   so  
	 $$dam_{L',P^i}(L'(u),L'(v)) \leq \ell-d=\ell'  \leq \ell'+ (S_{L'}(P^i)-n_jm-\frac{n_i+1}{2}-m+\tau'),$$  
	 which implies that $$ S_{L'}(P^i)-n_im \geq \frac{n_i+1}{2}+m+dam_{L',P^i}(L'(u),L'(v))-\ell'-\tau'.$$  
	Hence, \ref{T5} is   satisfied by $L'$.
	By induction hypothesis,   there exists a pair $(S,T)$, where $|S|=|T|=m-\tau$ such that for each $i \in \{0,1,2,3\}$, $dam_{L,P^i}(S,T) \leq S_L(P^i)-n_im$.  This completes the proof of this  claim.
\end{proof}

\begin{claim}
	\label{claim:s-all at most light}
	There does not exist simple pair $(D_u,D_v)$ such that $|D_u|=|D_v|=d \leq m-\tau$,  and
	for each $i \in \{0,1,2,3\}$,   $z^{(i)}=0$ or  $dam_{L,P^i}(D_u,D_v) \leq d$. 
\end{claim} 

\begin{proof}
	Assume   $(D_u,D_v)$ is such a simple pair. Let $L'$ be a new list assignment for $G$ with $L'(u)=L(u)-D_u$, $L'(v)=L(v)-D_v$,  $L'(v^i_1)=L(v^i_1)-D_u \cup D_v$ for $i=0,1,2$. 
	If $n_3=1$, then $L'(v^3_1)=L(v^3_1)-D_u \cup D_v$. Otherwise, $L'(v^3_1)= L(v^3_1)-D_u$, $L'(v^3_{n_3})=L(v^3_{n_3})-D_v$,  $L'(v^3_j)=L(v^3_j)$ where $1 < j < n_3$.
	
	\ref{T1}-\ref{T2} and \ref{T4} of Theorem \ref{thm:s-main-second-part-stronger}  are easily seen to be satisfied by $L'$, with $ \ell'= \ell - d$ and $\tau'=\tau+d$.

	It is obvious that \ref{T3} is  satisfied by $(L',P^i)$ when $n_i \geq 3$.
	Now we show that \ref{T3} is also satisfied when $n_i=1$. Assume $i \in \{0,1,2,3\}$ and $n_i=1$. If $dam_{L,P^i}(D_u,D_v) \leq d$, then $|L(v^i_1)-D_u \cup D_v| \ge |L(v^i_1)|-d$, so \ref{T3} is satisfied by $L'$. Assume $z^{(i)}=0$. Then for any couple $cc'$, $dam_{L,P^i}(c,c') \geq 1$, which implies that $dam_{L,P^i}(L(u),L(v)) \geq dam_{L,P^i}(D_u,D_v)+\ell-d$ (using Lemma \ref{lem-sum}). By \ref{T5},
	\begin{align*} 
	S_L(P^i) & \geq  \frac{n_i+1}{2}+m-\ell-\tau+dam_{L,P^i}(L(u),L(v))+n_im\\
	& \geq 2m+1-(\tau+d)+dam_{L,P^i}(D_u,D_v)\\
	& = 2m+1-\tau'+dam_{L,P^i}(D_u,D_v). 
	\end{align*}  
	As $n_i=1$ implies that $|L(v^i_1)|=S_L(P^i)$, $|L(v^i_1)-D_u\cup D_v|=|L(v^i_1)|-dam_{L,P^i}(D_u,D_v)=S_{L}(P^i)-dam_{L,P^i}(D_u,D_v) \geq 2m+1-\tau'$. So \ref{T3} is also satisfied by $L'$ in this case.
	
	Next, we show that \ref{T5} is satisfied by $L'$.
	By Lemma \ref{lem-sum}, $dam_{L,P^i}(L(u), L(v)) = dam_{L',P^i}(L'(u), L'(v)) + dam_{L,P^i}(D_u,D_v)$, so
	\begin{align}
	S_{L'}(P^i) - n_im  & =  S_L(P^i)-n_im -dam_{L,P^i}(D_u,D_v) \label{eq-s-claim4-2-1}  \\
	& \geq  m+ \frac{n_i+1}{2}+dam_{L,P^i}(L(u), L(v))-\ell-\tau- dam_{L,P^i}(D_u,D_v)  \nonumber  \\
	& =  m+ \frac{n_i+1}{2}+dam_{L',P^i}(L'(u), L'(v))-\ell'-\tau'. \label{eq-s-claim4-2-3}
	\end{align}
	Now it suffices to prove that $S_{L'}(P^i)-n_im \geq m+\frac{n_i+1}{2}-\tau'$. Indeed, if 
	 $z^{(i)}=0$, then  for each couple $cc'$, $dam_{L',P^i}(c,c') \geq 1$. By Lemma \ref{lem-sum}, $dam_{L',P^i}(L'(u),L'(v)) \geq \ell-d=\ell'$. By Inequality (\ref{eq-s-claim4-2-3}), we are done. By \ref{T5}, $S_{L}(P^i)-n_im \geq \frac{n_i+1}{2}+m-\tau$. If $dam_{L,P^i}(D_u,D_v) \leq d$,  then by Equality (\ref{eq-s-claim4-2-1}), $$S_{L'}(P^i)-n_im \geq  \frac{n_i+1}{2}+m-\tau-d= \frac{n_i+1}{2}+m-\tau'.$$
  	Therefore, \ref{T5} is  satisfied by $L'$.
	
	By induction,   there exists a pair $(S',T')$, where $|S'|=|T'|=m-\tau' = m-\tau-d$ such that for every $i$, $$dam_{L',P^i}(S',T') \leq S_{L'}(P^i)-n_im.$$  Let $S=S' \cup D_u$ and  $T=T' \cup D_v$.   As $S' \cap D_u = \emptyset$ and $T' \cap D_v = \emptyset$,  $dam_{L',P^i}(S',T'))=dam_{L,P^i}(S',T')$.  So we have $|S|=|T|=m-\tau$ and 
	\begin{align*} 
	dam_{L,P^i}(S,T) & \leq dam_{L,P^i}(D_u,D_v)+dam_{L,P^i}(S',T') \\
	& \leq dam_{L,P^i}(D_u,D_v)+  S_{L'}(P^i)-n_im\\
	& = S_{L}(P^i)-n_im. 
	\end{align*} 
	
	This completes the proof of Claim \ref{claim:s-all at most light}.
\end{proof}

Claim \ref{claim:s-bad simple pair} follows directly from the definitions and \ref{T5}.

\begin{claim}
	\label{claim:s-bad simple pair}
	If $(S,T)$ is a bad simple pair of size $m-\tau$ with respect to $(L, P^i)$, then 
	$dam_{L,P^i}(S,T) = 2a^{(i)}(S,T)+b^{(i)}(S,T) \geq \max\{2x^{(i)}+y^{(i)} +m+\frac{n_i+3}{2}-\ell-\tau, m+\frac{n_i+3}{2}-\tau\}$.
\end{claim}

\begin{claim}
	\label{claim:s-no simple pair of size 3 with damage at most 3}
	There is no  simple pair $(S_0,T_0)$ of size $3$ such that $dam_{L,P^i}(S_0,T_0) \leq 3$ for each $i \in \{0,1,2,3\}$. 	
\end{claim}
\begin{proof}
	Assume the claim is not true, and assume $S_0=\{c_1,c_2,c_3\}$, $T_0=\{c'_p,c'_q,c'_r\}$.
	If $m-\tau \geq 3$, then by Claim \ref{claim:s-all at most light}, this is a contradiction. Thus assume that $m-\tau \leq 2$. Recall that in the beginning of the proof of Theorem \ref{thm:s-main-second-part-stronger}, we argued that $m-\tau \geq 2$, so $m-\tau =2$.
	By \ref{T5}, $S_L(P^i)-n_im \geq \frac{n_i+1}{2}+m-\tau \geq m-\tau +1 \geq 3$. Then any simple pair $(S,T)$ of size $2$ with $S \subseteq S_0, T \subseteq T_0$  satisfies the theorem, a contradiction.
\end{proof}

\begin{claim}
	\label{claim:s-x or z=0 implies that beta=0}
	For each $i \in \{0,1,2,3\}$, $x^{(i)}=0$ or $z^{(i)}=0$ implies that $\beta(P^i)=0$. 
\end{claim}
\begin{proof}
	If $x^{(i)}=0$, then  $dam_{L,P^i}(S,T) \leq m-\tau$ for any simple pair $(S,T)$ of size $m-\tau$.  By \ref{T5},  $S_L(P^i)-n_im \geq \frac{n_i+1}{2}+m-\tau \geq m-\tau +1$. So $(S,T)$ is not bad with respect to $(L,P^i)$, hence $\beta(P^i)=0$.
	
	If $z^{(i)}=0$, then	$x^{(i)}+y^{(i)}= \ell$ and for any simple pair $(S,T)$ of size $m-\tau$, $a^{(i)}(S,T)+b^{(i)}(S,T)= m-\tau$. If $(S,T)$ is bad with respect to $(L,P^i)$, then by Claim \ref{claim:s-bad simple pair}, we have 
	\begin{align*}
	a^{(i)}(S,T) +m-\tau & = 2a^{(i)}(S,T)+b^{(i)}(S,T)  \\
	& \geq 2x^{(i)}+y^{(i)} +m+\frac{n_i+3}{2}-\ell-\tau\\
	& \geq  x^{(i)} + \ell+m+\frac{n_i+3}{2}-\ell-\tau\\
	& = x^{(i)}+m-\tau+ \frac{n_i+3}{2}.
	\end{align*}
	This implies that $a^{(i)}(S,T) \geq x^{(i)}+2$, in contrary to that  $a^{(i)}(S,T) \leq x^{(i)}$.  So any simple pair $(S,T)$ of size $m-\tau$ is not bad with respect to $(L,P^i)$, hence $\beta(P^i)=0$.
\end{proof}

\begin{claim}
	\label{claim:s-x and z is at least one}
	For each $j \in \{0,1,2,3\}$, $x^{(j)},z^{(j)} \geq 1$. 
\end{claim}

\begin{proof}
	Assume to the contrary, $x^{(j)}=0$ or $z^{(j)}=0$ for some $j$. 	Then $\beta(P^j)=0$ By Claim \ref{claim:s-x or z=0 implies that beta=0}. 
	For conveniece, we let $j=0$ below, but do not use the fact that $n_1=0$ so that the argument also works for $j=3$.

	We first show that $x^{(i)}, z^{(i)} \geq 1$ for $i \ne 0$. Indeed, if this fails for some $i$, then by Claim \ref{claim:s-x or z=0 implies that beta=0}, $\beta(P^i)=0$. Thus by Claim \ref{claim:s-bad simple pair}  and Observation \ref{obs:about main lemma} (setting $m-\tau=k$), $\sum_{i=0}^{3}\beta(P^i) < \binom{\ell}{m-\tau}$. So there  exists a simple pair of size $m-\tau$ which is not bad with respect to any $(L,P^i)$, a contradiction.

	Next we show that every couple is heavy (respectively, safe, light) for at most one of  $P^1, P^2, P^3$.  
	
	%First we consider the case that $x^{(0)}=0$.
	
	 Assume to the contrary, $c_jc'_j$ is heavy for two paths, say for both $P^1$ and $P^2$.  By Claim \ref{claim:s-all at least light},  $c_jc'_j$ is safe for $P^3$. As $x^{(3)} \geq 1$, there exists a couple $c_kc'_k$ which is heavy for $P^3$. 
	 Then for $D_u = \{c_j,c_k\}, D_v =\{c'_j,c'_k\})$, we have $dam_{L, P^{(i)}}(D_u, D_v) \ge 2$ for $i=1,2,3$, and for $i=0$,  either $x^{(0)}=0$,  or  $z^{(0)}=0$ which means that $dam_{L, P^{(0)}}(D_u, D_v) \ge 2$. In either case, it contradicts   Claim \ref{claim:s-all at least light}. 
	 
	Similarly, if $c_jc'_j$ is safe for $P^1$ and $P^2$, then by Claim \ref{claim:s-all at most light}, $c_jc'_j$ is not safe for $P^3$. As $z^{(3)} \geq 1$, so there exists a couple $c_kc'_k$ which is safe for $P^3$. Then for $D_u = \{c_j,c_k\}, D_v =\{c'_j,c'_k\}$, we have $dam_{L, P^{(i)}}(D_u, D_v) \le 2$ for $i=1,2,3$, and for $i=0$,  either $z^{(0)}=0$, or  $x^{(0)}=0$ which means that $dam_{L, P^{(0)}}(D_u, D_v) \le 2$. In either case, it contradicts  Claim \ref{claim:s-all at most light}.   
	
	Assume $c_jc'_j$ is light for $P^1$ and $P^2$. If $c_jc'_j$ is safe for $P^3$, then  $c_jc'_j$ is a simple pair contradicting  Claim \ref{claim:s-all at most light}. Otherwise,  
	$c_jc'_j$ is a simple pair contradicting  Claim \ref{claim:s-all at least light}.
 
Without loss of generality, assume $c_1c'_1$ is heavy for $P^1$, light for $P^2$ and safe for $P^3$.  Assume that 
	 $c_2c'_2$ is heavy for $P^2$, $c_3c'_3$ is heavy for $P^3$. Then $c_3c'_3$ is light for $P^1$ and safe for $P^2$, for otherwise, $(\{c_1,c_3\},\{c'_1,c'_3\})$ is a simple pair of size $2$ contradicting  Claim \ref{claim:s-all at most light}. 
	 Similarly, $c_2c'_2$ is light for $P^3$ and safe for $P^1$, for otherwise, $(\{c_2,c_3\},\{c'_2,c'_3\})$ is a simple pair of size $2$ contradicting  Claim \ref{claim:s-all at most light}. 
	
	If $x^{(0)}=0$, then $(\{c_1,c_2,c_3\},\{c'_1,c'_2,c'_3\})$ is a simple pair of size $3$ contradicting  Claim \ref{claim:s-no simple pair of size 3 with damage at most 3}. 
	
	Assume that $x^{(0)}\geq 1$ and $z^{(0)}=0$. If $m-\tau \geq 3$, then $(\{c_1,c_2,c_3\},\{c'_1,c'_2,c'_3\})$ is a simple pair of size $3$ contradicting  Claim \ref{claim:s-all at most light}. Assume $m-\tau=2$. By \ref{T5},
	 $$S_L(P^i)-n_im \ge \max\{dam_{L,P^i}(L(u),L(v))-\ell+3, 3\} \geq 3.$$ 
	Note that $x^{(0)}\geq 1$ and $z^{(0)}=0$ implies that $dam_{L,P^0}(L(u),L(v)) \geq 2+(\ell-1)=\ell+1$. Therefore, $S_L(P^0)-n_0m \geq 4$. Let $S=\{c_1,c_2\}$, $T=\{c'_1,c'_2\}$. Then $(S,T)$ is a pair of size $m-\tau=2$ satisfying Theorem \ref{thm:s-main-second-part-stronger}.
 
 This completes the proof of this claim.
\end{proof}
 
\begin{claim} 
	\label{claim:s-evey couple heavy or safe for at most two} 
	Every couple is heavy (respectively, light, safe) for at most two internal paths. 
\end{claim}

\begin{proof} 
	If there exists a couple $c_jc'_j$ which is light for at least three couples, then $c_jc'_j$ is counterexample with $d=1$ to either Claim \ref{claim:s-all at least light} or Claim \ref{claim:s-all at most light}.
	
	By Claim \ref{claim:s-all at least light}, every couple is heavy for at most three internal paths. If there exists a couple, say $c_jc'_j$ which is heavy for all the internal paths except $P^i$ for some $i \in \{0,1,2,3\}$. By Claim \ref{claim:s-x and z is at least one}, $x^{(i)} \geq 1$, there exists a heavy couple $c_kc'_k$ for $P^i$. Then $(\{c_j,c_k\},\{c'_j,c'_k\})$ is a simple pair of size $2$ contradicting  Claim \ref{claim:s-all at least light}. Thus every couple is heavy for at most two internal paths. Similarly, we can prove that every couple is safe for at most two internal paths. 
\end{proof}
	
\begin{claim}
	\label{claim:s-two fixed H two S}
	If a couple is heavy for exactly two internal paths, then it is safe for  the other two paths. 
\end{claim}	
\begin{proof}
	  Assume  the claim is not true and $c_0c'_0$ is  heavy   for two internal paths, and light for at least one internal path. If $c_0c'_0$ is light for two internal paths, then 
	  $c_0c'_0$ is a simple pair that contradicts  Claim \ref{claim:s-all at least light}. 
	  So $c_0c'_0$ is light for one internal path $P^i$ and safe for one internal path $P^j$.   Without loss of generality, assume $c_0 \in \hat{X}^{i}_{1} \cup \Lambda^{i}$. 
	  
	  As $x^{(j)} \ge 1$, there is a couple $c_1c'_1$ which is heavy for $P^j$. Note that $c_0c'_0$ is heavy for at least one internal path with only one vertex. So $c_0 \ne c'
	 _0$. If $c_1 \ne c'_1$, then by Observation \ref{obs-simple pair}, $(c_0,c'_1)$ is a simple pair. But  $dam_{L,P^i}(c_0,c'_1) \geq 1$ for each $i \in \{0, 1,2,3\}$,   contrary to Claim \ref{claim:s-all at least light}.  Thus $c_1=c'_1$. By Observation \ref{obs-couple}, $c_1c'_1$ can not be heavy for an internal path with only one vertex.  So $j=3$ and $n_3 \geq 3$. Thus we may assume that $c_0c'_0$ is heavy for $P^0$ and $P^1$, light for $P^2$ and safe for $P^3$, i.e., $i=2$.

  	Then  $c_1c'_1$ is safe for $P^2$, for otherwise
	 $(\{c_0,c_1\},\{c'_0,c'_1\})$ is a simple pair of size $2$ contradicting  Claim \ref{claim:s-all at least light}.

	By Claim \ref{claim:s-evey couple heavy or safe for at most two},  we may assume that $c_1c'_1$ is light for $P^0$, and is either light for $P^1$ or safe for $P^1$.
	
	By Claim \ref{claim:s-x and z is at least one}, $x^{(2)} \geq 1$. Let  $c_2c'_2$ be a couple which is heavy for $P^{2}$. Then $c_2 \neq c'_2$. 
	
	We claim that $c_2c'_2$ is safe for $P^{3}$. Otherwise  $c_2c'_2$ is heavy or light for $P^{3}$. Without loss of generality, assume that $c_2 \in  \hat{X}^{3}_{1} \cup \Lambda^{3}$. By Observation \ref{obs-simple pair}, $(c_2,c'_0)$ is a simple of size $1$ and $dam_{L,P^i}(c_2,c'_0) \geq 1$, a contradiction to Claim \ref{claim:s-all at least light}.  
	
	Recall that we assumed that $c_0 \in \hat{X}^{2}_{1} \cup \Lambda^{2}$.
	Hence $dam_{L,P^i}(c_0,c'_2) \geq 1$ for $i=0,1$, $dam_{L,P^2}(c_0,c'_2)=2$ and $dam_{L,P^3}(c_0,c'_2) =0$. So if $c_1c'_1$ is light for $P^1$, then $(\{c_0,c_1\},\{c'_1,c'_2\})$ is a simple pair (by Observation \ref{obs-simple pair}) of size $2$ contradicting  Claim \ref{claim:s-all at least light}. Therefore, $c_1c'_1$ is safe for $P^1$.
	
Then $c_2c'_2$ must be heavy for $P^0$,  for otherwise $(\{c_1,c_2\}, \{c'_1,c'_2\})$ is a simple pair of size $2$ contradicting  Claim \ref{claim:s-all at most light}. 

Thus $c_2c'_2$ is either light for $P^1$ or safe for $P^1$.
	
	By Claim \ref{claim:s-x and z is at least one}, $z^{(0)} \geq 1$. Let $c_3c'_3$ be a couple which is safe for $P^{0}$. 
	
	We claim that $c_3c'_3$ is heavy for at least one of $P^{1}$ and $P^{2}$. Otherwise $c_3c'_3$ is heavy for $P^{3}$ by Claim \ref{claim:s-all at most light}. If $c_3c'_3$ is safe for $P^{2}$, then $(\{c_2,c_3\},\{c'_2,c'_3\})$ is a simple pair of size $2$ contradicting  Claim \ref{claim:s-all at most light}. 
	
	If $c_3c'_3$ is safe for $P^{1}$, then $(\{c_0,c_3\},\{c'_0,c'_3\})$ is a simple pair which contradicts Claim \ref{claim:s-all at most light}. So $c_3c'_3$ is light for $P^{1}$ and $P^{2}$. 
	
	Recall that $(c_0,c'_2)$ is a simple pair satisfying that $dam_{L,P^i}(c_0,c'_2)=2$ for   $i \in \{0,2\}$, $dam_{L,P^1}(c_0,c'_2)\geq 1$ and $dam_{L,P^3}(c_0,c'_2)=0$. Thus $(\{c_0,c_3\},\{c'_2,c'_3\})$ is a simple pair of size $2$ contradicting  Claim \ref{claim:s-all at least light}.
	
	 This completes the proof of the claim that $c_3c'_3$ is heavy for at least one of $P^{1}$ and $P^{2}$. Hence $c_3\neq c'_3$. If $c_3c'_3$ is safe for $P^3$,  then $(\{c_1,c_3\},\{c'_1,c'_3\})$ is a simple pair of size $2$ contradicting  Claim \ref{claim:s-all at most light}. So $c_3c'_3$ is not safe for $P^3$.  Thus $c_3 \in \hat{X}^3_{1}\cup \Lambda^3$ or $c'_3 \in \hat{X}^3_{n_3}\cup \Lambda^3$ (or both).    If $c'_3 \in \hat{X}^3_{n_3}\cup \Lambda^3$, then $(c_0,c'_3)$ is a simple pair of size $1$  contradicting  Claim \ref{claim:s-all at least light}. So $c_3c'_3$ is light for $P^3$ and $c_3 \in \hat{X}^3_{1}\cup \Lambda^3$. 
	 
	 If $c_3c'_3$ is   heavy for $P^1$, then $(c_3,c'_2)$ is a  simple pair of size $1$  contradicting  Claim \ref{claim:s-all at least light}. If $c_3c'_3$ is heavy for $P^2$,   then $(c_3,c'_0)$ is   a simple pair of size $1$  contradicting  Claim \ref{claim:s-all at least light}.  

	This completes the proof of Claim \ref{claim:s-two fixed H two S}.
\end{proof}

\begin{claim}
	\label{claim:s-not heavy for two paths with one being P3}
	No couple is heavy for two internal paths with one being $P^3$. 
\end{claim}	
\begin{proof}
Assume to the contrary that $c_0c'_0$ is heavy for $P^0$ and $P^3$. As $n_0=1$, by Observation \ref{obs-couple}, $c_0\neq c'_0$. By Claim \ref{claim:s-two fixed H two S}, $c_0c'_0$ is safe for $P^1$ and $P^2$. 
	
	By Claim \ref{claim:s-x and z is at least one}, $x^{(1)}\geq 1$. Let $c_1c'_1$ be  a couple which is heavy for $P^1$. Then $c_1\neq c'_1$. Observe that $c_1c'_1$ is   safe for $P^2$, for otherwise, we may assume   $c_1 \in \hat{X}^2_1 \cup \Lambda^2$, and hence $(c_1,c'_0)$ is a simple pair (by Observation \ref{obs-simple pair}) of size $1$ contradicting  Claim \ref{claim:s-all at least light}. 
	 Similarly,  there exists a couple  $c_2c'_2$, which is heavy for $P^2$ and safe for $P^1$, and $c_2 \neq c'_2$.
	
	At least one of $c_1c'_1$ and $c_2c'_2$ is heavy for $P^0$ or $P^3$, for otherwise, $(\{c_2,c_3\},\{c'_2,c'_3\})$ is a simple pair of size $2$ contradicting  Claim \ref{claim:s-all at most light}. Without loss of generality, assume that $c_1c'_1$ is heavy for $P^0$ or $P^3$.   For convenience,    assume that $c_1c'_1$ is heavy for $P^0$, and we will not use the fact that $n_1 =1$. By Claim \ref{claim:s-two fixed H two S}, $c_1c'_1$ is safe for $P^3$. 
	
	As $c_2\neq c'_2$,  $c_2c'_2$ is safe for $P^3$, since otherwise, without loss of generality, we assume $c_2 \in \hat{X}^3_1 \cup \Lambda^3$. Then $(c_2,c'_1)$ is a simple pair (by Observation \ref{obs-simple pair}) of size $1$ contradicting  Claim \ref{claim:s-all at least light}. Similarly, we have $(c_2,c'_2)$ is heavy for $P^0$, for otherwise,  without loss of generality, we assume $c_2 \notin \hat{X}^0_1 \cup \Lambda^0$.  Then $(c_2,c'_1)$ is a  simple pair (by Observation \ref{obs-simple pair}) of size $1$ contradicting  Claim \ref{claim:s-all at most light}.
	
	By Claim \ref{claim:s-x and z is at least one}, $z^{(0)}\geq 1$. Let $c_3c'_3$ be a couple which is safe for $P^{0}$. Note that $(c_1,c'_2)$ is a simple pair of size $1$ which is heavy for $P^0$, light for $P^1$ and $P^2$, and safe for $P^3$.		
	Therefore,  if $c_3c'_3$ is neither heavy for $P^1$ nor for $P^2$, then $(\{c_1,c_3\},\{c'_2,c'_3\})$ is a simple pair of size $2$ contradicting  Claim \ref{claim:s-all at most light}. Thus without loss of generality, assume that $c_3c'_3$ is heavy for $P^1$. If $c_3c'_3$ is also heavy for $P^2$, then by Claim \ref{claim:s-two fixed H two S}, $c_3c'_3$ is safe for $P^3$. But then  $(\{c_0,c_3\},\{c'_0,c'_3\})$ is a simple pair of size $2$ contradicting  Claim \ref{claim:s-all at least light}. So $c_3c'_3$ is not heavy for $P^2$. If $c_3c'_3$ is   safe for $P^2$, then $(\{c_2,c_3\},\{c'_2,c'_3\})$ is a simple pair of size $2$ contradicting  Claim \ref{claim:s-all at most light}. Thus $c_3c'_3$ is light for $P^2$. Without loss of generality, assume that $c_3 \notin \hat{X}^2_1 \cup \Lambda^2$. But then $(c_3,c'_2)$ is a simple pair of size $1$ contradicting  Claim \ref{claim:s-all at most light}.  
	
	This completes the proof of Claim \ref{claim:s-not heavy for two paths with one being P3}.	
\end{proof}

\begin{claim}
	\label{claim:s-every couple is heavy for at most one path}
	Every couple is heavy for exactly one internal path. 
\end{claim}	
\begin{proof}
	By Claim \ref{claim:s-all at most light},   every couple is heavy for at least one internal path.
	Suppose to the contrary, $c_0c'_0$ is heavy for two internal paths. By Claim \ref{claim:s-not heavy for two paths with one being P3}, we may assume that $c_0c'_0$ is heavy for $P^0$ and $P^1$. By Claim \ref{claim:s-two fixed H two S}, $c_0c'_0$ is safe for both $P^2$ and $P^3$.  
 	
	By Claim \ref{claim:s-x and z is at least one}, $x^{(2)}\geq 1$. Let $c_1c'_1$ be a couple which is heavy for $P^2$. Then $c_1 \neq c'_1$. Note that $c_1c'_1$ must be safe for $P^3$, for otherwise,   we assume that $c_1 \in \hat{X}^3_1 \cup \Lambda^3$. Then $(c_1,c'_0)$ is a simple pair (By Observation \ref{obs-simple pair}) of size $1$ contradicting  Claim \ref{claim:s-all at least light}.  
	
	As $x^{(3)} \geq 1$,  there exists a couple  $c_2c'_2$  which is heavy for $P^3$. By Claim \ref{claim:s-not heavy for two paths with one being P3}, we know that $c_2c'_2$ is not heavy for any   of $P^0$, $P^1$ and $P^2$.  
	
	We first claim that $c_1c'_1$ is heavy for exactly one of $P^0$ and $P^1$. Suppose this is not true. By Claim \ref{claim:s-evey couple heavy or safe for at most two}, $c_1c'_1$ can not be safe for three paths, so $c_1c'_1$ is light for at least one of $P^0$ and $P^1$, without loss of generality, say $P^0$, and assume that $c_1 \notin \hat{X}^0_1 \cup \Lambda^0$. If $c_1c'_1$ is safe for $P^1$, then $(c_1,c'_0)$ is a simple pair of size $1$ contradicting  Claim \ref{claim:s-all at most light}. Thus $c_1c'_1$ is also light for $P^1$. Recall that $c_2c'_2$ is heavy for none of $P^0$, $P^1$ and $P^2$. If $c_2c'_2$   is safe for $P^2$,   then $(\{c_1,c_2\},\{c'_1,c'_2\})$ is a  simple pair of size $2$ contradicting  Claim \ref{claim:s-all at most light}. So $c_2c'_2$ is light for $P^2$.  By Claim \ref{claim:s-all at least light}, $c_2c'_2$ is safe for at least one of $P^0, P^1$.  Assume that it is safe for $P^1$. Then  $c_2c'_2$ is light for $P^0$, for otherwise $(\{c_0,c_2\},\{c'_0,c'_2\})$ is a  simple pair of size $2$ contradicting  Claim \ref{claim:s-all at most light}. Recall that $c_1 \neq c'_1$, and we assumed that $c_1 \notin \hat{X}^0_1 \cup \Lambda^0$. 
	So $(c_1,c'_0)$ is a simple pair of size $1$ such that it is light for $P^0$ and $P^2$,   safe for $P^3$. Hence $(\{c_1,c_2\},\{c'_0,c'_2\})$ is a simple pair of size $2$ contradicting  Claim \ref{claim:s-all at most light}. 
	
Without loss of generality, we assume that $c_1c'_1$ is heavy for $P^0$, by Claim \ref{claim:s-two fixed H two S}, $c_1c'_1$ is safe for $P^1$. Note that $(c_0,c'_1)$ is a  simple pair of size $1$ which is heavy for $P^0$, light for $P^1$ and $P^2$, and safe for $P^3$. If $c_2c'_2$ is safe for $P^0$, then 
	$(\{c_0,c_2\},\{c'_1,c'_2\})$ is a simple pair of size $2$ contradicting  Claim \ref{claim:s-all at most light}. Hence $c_2c'_2$ is
	light for $P^0$.  On the other hand,  by Claim \ref{claim:s-all at least light}, $c_2c'_2$ is safe for some $P^i$. Without loss of generality, we assume that $c_2c'_2$ is safe for $P^1$.

	By Claim \ref{claim:s-x and z is at least one}, $z^{(0)}\geq 1$, we may assume that $c_3c'_3$ is safe for $P^{0}$. Note that $c_3c'_3$ is heavy for at least one of $P^1$ or $P^2$, for otherwise, $(\{c_0,c_3\},\{c'_1,c'_3\})$ is a simple pair of size $2$ contradicting  Claim \ref{claim:s-all at most light} .  So $c_3 \neq c'_3$, and  by Claim \ref{claim:s-not heavy for two paths with one being P3}, $c_3c'_3$ is not heavy for $P^3$.
	
	First assume that $c_3c'_3$ is heavy for $P^1$. If $c_3c'_3$ is also heavy for $P^2$, then by Claim \ref{claim:s-two fixed H two S}, $c_3c'_3$ is safe for $P^3$.	
	But then  $(\{c_0,c_2\},\{c'_2,c'_3\})$ is a simple pair of size $2$ contradicting  Claim \ref{claim:s-all at most light} (recall that by Claim \ref{claim:s-not heavy for two paths with one being P3}, $c_2c'_2$ is not heavy for $P^2$ as it is already heavy for $P^3$ ). So $c_3c'_3$ is not heavy for $P^2$. Thus without loss of generality, we may assume that $c_3 \notin \hat{X}^2_1 \cup \Lambda^2$. Then   $(c_3,c'_1)$ is a simple pair of size $1$ contradicting  Claim \ref{claim:s-all at most light}. So $c_3c'_3$ is not heavy for $P^1$. 
	
	Therefore, $c_3c'_3$ is heavy for $P^2$ but not heavy for $P^1$. Thus without loss of generality,  assume that $c_3 \notin \hat{X}^1_1 \cup \Lambda^1$. But then we have that $(c_3,c'_0)$ is a simple pair of size $1$ which is not heavy for any internal paths, a contradiction to Claim \ref{claim:s-all at most light}.
	
	This completes the proof of Claim \ref{claim:s-every couple is heavy for at most one path}.	
\end{proof}

\begin{claim}
	\label{claim:s-every couple is safe for at most one path}
	Every couple is safe for exactly one internal path. 
\end{claim}	
\begin{proof}
	By Claim \ref{claim:s-all at least light}, we know that every couple is safe for at least one internal path.
	 
	Assume to the contrary that $c_0c'_0$ is  safe for $P^{2}$ and $P^{3}$. As every couple is heavy for exactly one internal path, we may assume that $c_0c'_0$ is heavy for $P^{0}$, light for $P^{1}$.  Note that path $P^3$ is different from the other paths, as $n_i=1$ for $i=0,1,2$ and $n_3 $ can be greater than $1$. However, the argument below does not use this difference.

	By Claim \ref{claim:s-x and z is at least one}, $x^{(2)} \geq 1$, $x^{(3)} \geq 1$, and by Claim \ref{claim:s-every couple is heavy for at most one path}, every couple is heavy for exactly one path, thus we assume that $c_1c'_1$ is heavy for $P^{2}$ and $c_2c'_2$ is heavy for $P^{3}$. 
	
	If $c_1c'_1$ is safe for $P^{3}$ and $c_2c'_2$ is safe for $P^{2}$, then $(\{c_1,c_2\},\{c'_1,c'_2\})$ is a simple pair of size $2$ contradicting  Claim \ref{claim:s-all at most light}. So without loss of generality, we assume that $c_1c'_1$ is light for $P^{3}$. 
	
	By Claim \ref{claim:s-all at most light}, both $c_1c'_1$ and $c_2c'_2$ are light for $P^{0}$  by considering $(\{c_0,c_1\},\{c'_0,c'_1\})$ and $(\{c_0,c_2\},\{c'_0,c'_2\})$, respectively. Consequently, $c_1c'_1$ is safe for $P^{1}$ since otherwise $c_1c'_1$ is not safe for any internal path, a contradiction.
 	
	By Claim \ref{claim:s-x and z is at least one}, $x^{(1)} \geq 1$, there exists a couple, say $c_3c'_3$, which is heavy for $P^{1}$. By Claim \ref{claim:s-every couple is heavy for at most one path}, $c_3c'_3$ is not heavy for any other paths. Thus  $c_3c'_3$ is light for $P^{2}$, for otherwise, $(\{c_1,c_3\},\{c'_1,c'_3\})$ is a simple pair of size $2$ contradicting  Claim \ref{claim:s-all at most light}. 	
	
	If $c_3c'_3$ is safe for $P^{0}$, then  
	$(\{c_0,c_1,c_3\},\{c'_0,c'_1,c'_3\})$ is a simple pair of size $3$ which contradicts Claim \ref{claim:s-no simple pair of size 3 with damage at most 3}. So $c_3c'_3$ is light for $P^{0}$, which implies that $c_3c'_3$ is safe for $P^{3}$.
	Similarly,  $c_2c'_2$ is light for $P^{2}$, as otherwise $(\{c_1,c_2,c_3\},\{c'_1,c'_2,c'_3\})$ is a simple pair of size $3$ which contradicts Claim \ref{claim:s-no simple pair of size 3 with damage at most 3}.  This implies that $c_2c'_2$ is safe for $P^{1}$. But then $(\{c_2,c_3\},\{c'_2,c'_3\})$ is a simple pair of size $2$ which contradicts Claim \ref{claim:s-all at most light}.
	
	This completes the proof of Claim \ref{claim:s-every couple is safe for at most one path}.
\end{proof}
	
	 By Claim \ref{claim:s-x and z is at least one}, $x^{(i)} \geq 1$ for each $i \in\{0,1,2,3\}$. Without loss of generality, assume that $c_0c'_0$ is heavy for $P^3$, light for $P^1$ and $P^2$, safe for $P^0$. Also, we assume that $c_1c'_1$ is heavy for $P^0$, $c_2c'_2$ is heavy for $P^2$, $c_3c'_3$ is heavy for $P^1$. Observe that $c_1c'_1$ must be light for $P^3$, since otherwise, $c_0c'_0$ and $c_1c'_1$ comprise a simple pair of size $2$ contradicting Claim \ref{claim:s-all at most light}. By Claim \ref{claim:s-every couple is safe for at most one path}, $c_1c'_1$ must be safe for exactly one of $P^1$ and $P^2$,   say $P^2$, and then it is light for $P^1$. This implies that $c_2c'_2$ is light for $P^0$, for otherwise  $c_1c'_1$ and $c_2c'_2$ comprise a simple pair of size $2$ contradicting Claim \ref{claim:s-all at most light}. Similarly, $c_2c'_2$ is light for $P^3$ by considering Claim \ref{claim:s-no simple pair of size 3 with damage at most 3} and the three couples $c_0c'_0$, $c_1c'_1$ and $c_2c'_2$.  Consequently, $c_2c'_2$ is safe for $P^1$. Again by these techniques, $c_3c'_3$ is light for $P^2$ by considering  Claim \ref{claim:s-all at most light} and the two couples $c_2c'_2$, $c_3c'_3$, and light for $P^0$ by considering  Claim \ref{claim:s-no simple pair of size 3 with damage at most 3} and the three couples $c_1c'_1$, $c_2c'_2$ and $c_3c'_3$. So $c_3c'_3$ is safe for $P^3$. See Table \ref{table2}.
	
	\begin{table}[H]
		\centering
		\begin{tabular}{c|c|c|c|c|c} 
			\bottomrule[1pt] \hline $L(u)$ & $c_0$ & $c_1$ & $c_2$ & $c_3$ &   $\cdots$   \\	
			\hline $P^0$  & safe  & heavy    & light    & light        &   $\cdots$      \\
			\hline $P^1$  & light & light    & safe     & heavy        &  $\cdots$      \\
			\hline $P^2$  & light & safe     & heavy    & light         &    $\cdots$    \\ 
			\hline $P^3$  & heavy & light    & light    & safe         &    $\cdots$    \\	 
			\hline $L(v$) & $c'_0$  & $c'_1$ & $c'_2$ & $c'_3$ &   $\cdots$    \\	
			\toprule[1.5pt]
		\end{tabular} 
		\caption{$dam_{L,P^i}(c_i,c'_i)$}
		\label{table2}
	\end{table}

	If  $m-\tau\geq 4$, then $(\{c_0,c_1,c_2,c_3\},\{c'_0,c'_1,c'_2,c'_3\})$ is a simple pair of size $4$ contradicting  Claim \ref{claim:s-all at most light}. So $m-\tau \leq 3$. Recall that $m-\tau \geq 2$, so $m-\tau=2$ or $m-\tau=3$.
		
	By \ref{T5}, $S_L(P^i)-n_im \geq \frac{n_i+1}{2}+m-\tau \geq m-\tau +1$.
	If $m-\tau =2$, then $S_L(P^i)-n_im \geq 3 $, we let $S=\{c_0,c_1\}$ and $T=\{c'_0,c'_1\}$; If $m-\tau =3$, then $S_L(P^i)-n_im \geq 4 $, we let $S=\{c_0,c_1,c_2\}$ and $T=\{c'_0,c'_1,c'_2\}$. In either case, we find a simple pair of size $m-\tau$ which satisfies the theorem, a contradiction.

	This finishes the proof of the Theorem \ref{thm:s-main-second-part-stronger}.
\end{proof}

\begin{corollary}
	$\Theta_{2,2,2,2p}$ is $(2m+1,m)$-choosable
\end{corollary}
\begin{proof}
	By setting $\ell =2m+1$ and $\tau=0$ in Theorem \ref{thm:s-main-second-part-stronger}, we know that $\Theta_{2,2,2,2p}$ is $(2m+1,m)$-choosable. Indeed, assume $L$ is a \bl{$(2m+1)$-list} assignment of $G=\Theta_{2,2,2,2p}$. By  Lemma \ref{second-lower-bound-for-slp}, $S_{L}(P^i) \geq \frac{n_i+1}{2}(2m+1)$, namely, $S_L(P^i)-n_im \geq \frac{n_i+1}{2}+m=\frac{n_i+1}{2}+m-\tau$. By Lemma \ref{first-lower-bound-for-slp}, $S_{L}(P^i) -\frac{n_i+1}{2}(2m+1)+(2m+1) \geq |\hat{X}^i_1|+|\hat{X}^i_n|+|\Lambda^i| \geq dam_{L,P^i}(L(u),L(v))$, so $S_L(P^i)-n_im \geq \frac{n_i+1}{2}+m+dam_{L,P^i}(L(u),L(v))-(2m+1)=\frac{n_i+1}{2}+m+dam_{L,P^i}(L(u),L(v))-\ell-\tau$. So \ref{T5} holds. Observe that $L$, $\ell$, $\tau$ also satisfies \ref{T1}-\ref{T4}. Therefore there exist two   sets   $S \subset L(u)$, $T \subset L(v)$ such that $|S|=|T|=m$ and $dam_{L,P^i}(S,T) \leq S_L(P^i) -n_im$ for $i=0,1,2,3$. Hence $G$ is $(2m+1,m)$-choosable.
\end{proof}

\section{Proof of Lemma \ref{main-lemma}}
\label{sec-main-lemma}

This section  proves Lemma \ref{main-lemma}.  I.e., 
\begin{equation}
\label{orginal-def-lemma} 
F(x,y)=\sum \binom{x}{a}\binom{y}{b}\binom{\ell-x-y}{k-a-b} \leq \frac 12 {\ell \choose k},
\end{equation}
where $x+y \le \ell, 2x+y \le \ell+k-1$, $k\geq 1$, $\ell \geq k+1$, and the summation is over  non-negative integer pairs $(a,b)$ for which $0 \leq a \leq x$, $0 \leq b \leq y$, $a+b \le k$ and  $2a+b \geq \max\{2x+y+k+1-\ell, k+1\}$.  Moreover,  we will show that the equality holds if and only if $\ell$ is even, $k$ is odd, and $x=\frac{\ell}{2}$, $y=0$.

Note that $a+b \le k$ and  $2a+b \geq  k+1$ implies that $a \ge 1$.

In the sequel, we define 
\begin{equation}
\label{eq-binomal}
\binom{p}{q}_+=
\begin{cases}
\binom{p}{q} & \text{if $p \geq q \geq 0$,}\\
0 & \text{if $q<0$ or $p <q$.}
\end{cases}
\end{equation}

For convenience, we allow $p<q$ or $q<0$ in the binormial coefficient in the summations below.   It is easy to check that in these cases,  either the pair $(a,b)$ does not lie in the range of the summation, and hence  contributes $0$ to the summations, or by extending  the equality $\binom{p+1}{q}=\binom{p}{q}+\binom{p}{q-1}$ to $q=0$. For the readability, we suppress the index `$+$'. 

%First, we analyze the monotonicity  of $F(x,y)$ about $y$ when $x$ is fixed, say $x=x_0$.  

The following lemma is proved in \cite{XuZhu2020} (Lemma 18 in \cite{XuZhu2020}, where the  parameter $2k$ is changed to $k$, but the proof still works).

\begin{lemma}
	\label{monotonous}
	Assume $x=x_0$ is fixed.      
	\begin{enumerate}[label= {(\arabic*)}]
		\item \label{mono-down} If $y \ge \ell-2x_0$, then $F(x_0,y+1)\leq F(x_0,y)$.
		\item \label{mono-up} If $ y < \ell -2x_0$, then $F(x_0,y) \leq F(x_0,y+1)$.
	\end{enumerate} 
\end{lemma}

We consider two cases.

\bigskip
\noindent
{\bf Case 1.} $x \leq \lfloor\frac{\ell}{2}\rfloor$.

\noindent 
 By Lemma \ref{monotonous}, $F(x,y) \leq  F(x,\ell-2x)$. 
So it suffices to show that $F(x,\ell-2x) \leq  \frac 12 \binom{\ell}{k}$.
Recall that (by Equality (\ref{orginal-def-lemma}))
$$F(x,\ell-2x)  = \sum\limits_{t=k+1}^{2k}\sum\limits_{2a+b =t}\binom{x}{a}\binom{\ell-2x}{b}\binom{x}{k-a-b}=\sum\limits_{t=k+1}^{2k} C(t,x),$$ 
where  $$C(t,x) =  \sum\limits_{2a + b=t}  \binom{x}{a} \binom{\ell-2x}{b} \binom{x}{k -a-b} = \sum\limits_{2a \leq t}  \binom{x}{a} \binom{\ell-2x}{t-2a} \binom{x}{k +a - t}.$$ 
Note that for any $0 \le x \le \frac{\ell}{2}$, $\sum_{t=0}^{2k}C(t,x) = \binom{\ell}{k}$. 

For  $0 \leq t \leq k$,  $$	C(t,x)  = \sum\limits_{2a \leq t} \binom{x}{a} \binom{\ell-2x}{t-2a} \binom{x}{k +a - t}  =  \sum\limits_{2a' \leq 2k-t} \binom{x}{a'} \binom{\ell-2x}{2k-t-2a'} \binom{x}{a'+t-k} = C(2k -t,x), $$
where $a'=k+a-t$.

When $1 \leq x \leq \lfloor\frac{\ell}{2}\rfloor$, 
\begin{equation}
\label{eq-function-F(x,y)} 
F(x,\ell-2x) = \sum_{t=k+1}^{2k}C(t,x) = \frac {\binom{\ell}{2k} - C(k,x)}{2}. 
\end{equation}

\begin{lemma}
	\label{lemma:C(t,x) is at least 0}
	$C(k, x)\geq 0$ when $1 \leq x \leq \lfloor\frac{\ell}{2}\rfloor$ and the equality holds if and only if $\ell=2x$ and $k$ is odd. 
\end{lemma} 
\begin{proof}
	If $x=\frac{\ell}{2}$ and $k$ is odd, then $y=\ell-2x=0$, which implies that $b=0$ as $0 \leq b \leq y$. Therefore, as $2a$ is even, we have  
	$C(k, x) = \sum_{2a=k}\binom{x}{a}^2=0$. 

	Assume $\ell \ne 2x$ or $\ell=2x$ and $k$ is even.
 
	First assume that $x \geq \lfloor\frac{k}{2}\rfloor$. As $x \leq \lfloor\frac{\ell}{2}\rfloor$, so $\ell-2x \geq 0$. Note that 
	$$k-2\lfloor\frac{k}{2}\rfloor = \begin{cases}
	0 & \text{if $k$ is even,}\\
	1 & \text{if $k$ is odd.}
	\end{cases}$$
	By Equality (\ref{eq-binomal}), $\binom{\ell-2x}{k-2\lfloor\frac{k}{2}\rfloor}=0$ if and only if $\ell-2x < k-2\lfloor\frac{k}{2}\rfloor$, i.e. $\ell-2x=0$ and $k-2\lfloor\frac{k}{2}\rfloor=1$, which means that $\ell=2x$ and $k$ is odd. So $\binom{\ell-2x}{k-2\lfloor\frac{k}{2}\rfloor}>0$ and 
	\begin{equation*} 
		C(k, x) \geq \binom{x}{\lfloor\frac{k}{2}\rfloor}^2\binom{\ell-2x}{k-2\lfloor\frac{k}{2}\rfloor} \geq 1.
	\end{equation*} 
 
Next assume that $1 \leq x \leq \lfloor\frac{k}{2}\rfloor-1$, then $k-2x >0$. Recall that $\ell \geq k+1$, so $\ell-2x > k-2x$. Hence,  
\begin{equation*} 
C(k, x) \geq \binom{x}{x}^2\binom{\ell-2x}{k-2x} \geq 1.
\end{equation*} 
	This completes the proof of Lemma \ref{lemma:C(t,x) is at least 0}.
\end{proof}
 
 Lemma \ref{lemma:C(t,x) is at least 0} and Inequality (\ref{eq-function-F(x,y)}) implies that when $1 \leq x \leq \lfloor\frac{\ell}{2}\rfloor$, $F(x,y) \leq \frac{1}{2}\binom{\ell}{k}$, and the equality holds if and only if $\ell=2x$ and $k$ is odd.

\bigskip
\noindent
{\bf Case 2.} $x \geq \lceil \frac{\ell}{2} \rceil$.

\noindent In this case,  $y \geq 0 \geq  \ell-2x$. By  Lemma \ref{monotonous}\ref{mono-down}, $F(x,y) \leq F(x,0)$. 

Note that in this case,  $b=y=0$ and $2x+y+k+1-\ell \geq  k+1$, so $2a+b=2a \geq 2x+y+k+1-\ell = 2x+k+1-\ell$. For brevity, let $p(x)=x+\lceil\frac{k+1-\ell}{2}\rceil$. Then $a \geq p(x)$ and
\begin{equation}
\label{eq-F(X,0)} 
F(x,0) = \sum\limits_{i=p(x)}^{k}\binom{x}{i}\binom{\ell-x}{k-i}. 
\end{equation}

\begin{lemma}
  $F(x,y) \leq F(\lceil \frac{\ell}{2} \rceil,0)$ whenever $x \geq \lceil \frac{\ell}{2} \rceil$.
\end{lemma}
\begin{proof}
We first prove that $F(x,0) \geq  F(x+1,0)$.  
 Let $\Delta  = F(x,0) -F(x+1,0)$, then
\begin{equation}
\label{delta3}
\Delta=\sum\limits_{i=p(x)}^{k}\binom{x}{i}\binom{\ell-x}{k-i}-\sum\limits_{i=p(x+1)}^{k}\binom{x+1}{i}\binom{\ell-1-x}{k-i}.
\end{equation}
Note that $p(x+1)=p(x)+1$. Using equalities $\binom{x+1}{i} = \binom{x}{i}+\binom{x}{i-1}$ and $\binom{\ell-x}{k-i}=\binom{\ell-x-1}{k-i}+\binom{\ell-x-1}{k-i-1}$,  and cancel the term $ \sum\limits_{j=p(x)+1}^{k-1}\binom{x}{i}\binom{\ell-x-1}{k-i}$, we have

$$\Delta =\binom{x}{p(x)}\binom{\ell-x}{k-p(x)}+\sum\limits_{i=p(x)+1}^{ k}\binom{x}{i}\binom{\ell-1-x}{k-1-i}  - \sum\limits_{i=p(x)+1}^{k}\binom{x}{i-1}\binom{\ell-1-x}{k-i}.$$
When $i=k$ in the first sum above, we have $\binom{\ell-1-x}{-1}=0$. Writing the last sum in the  equality above as $\sum\limits_{i=p(x)}^{k-1}\binom{x}{i}\binom{\ell-1-x}{k-1-i}$,   we have 
\begin{equation*}
	\Delta  = \binom{x}{p(x)}\binom{\ell-x}{k-p(x)} - \binom{x}{p(x)}\binom{\ell-1-x}{k-p(x)-1}  
  =\binom{x}{p(x)}\binom{\ell-1-x}{k-p(x)} \geq 0.
\end{equation*}
%Note that in above inequality, we use the fact that $\ell-1-x \geq k-p(x)$, which follows from that $x \geq \lfloor\frac{\ell}{2}\rfloor+1$.  
So, $F(x,y) \leq F(\lceil \frac{\ell}{2} \rceil, 0)$.
\end{proof}

If $\ell$ is even, then by Lemma \ref{monotonous} and the case that $x \leq \lfloor\frac{\ell}{2}\rfloor$, we have $ F(\lceil \frac{\ell}{2} \rceil,0) = F( \lfloor\frac{\ell}{2}\rfloor,\ell- 2\lfloor\frac{\ell}{2}\rfloor) \leq \frac{1}{2}\binom{\ell}{k}$, the equality holds if and only if $\ell=2x$ and $k$ is odd.  

In the rest of the proof, we assume that $\ell$ is odd, and we prove that $F(\lceil \frac{\ell}{2} \rceil, 0) \leq F(\lfloor\frac{\ell}{2}\rfloor,1)$, i.e., $F(\frac{\ell+1}{2},0) \leq F(\frac{\ell-1}{2},1)$, which implies that Lemma \ref{main-lemma} when $x \geq \lceil \frac{\ell}{2} \rceil$, as by the first case, $F(\frac{\ell-1}{2},1) < \frac{1}{2}\binom{\ell}{k}$.

 Note that in $F(\frac{\ell-1}{2},1)$, the summation is over $b=0$ and $1$. Using $\binom{\frac{\ell-1}{2}}{-1}=0$ and writing $\sum\limits_{i= \lceil\frac{k}{2}\rceil}^{k-1}\binom{\frac{\ell-1}{2}}{i}\binom{\frac{\ell-1}{2}}{k-1-i}$ as $\sum\limits_{i= \lceil\frac{k}{2}\rceil+1}^{k}\binom{\frac{\ell-1}{2}}{i-1}\binom{\frac{\ell-1}{2}}{k-i}$, we have  
 \begin{align*} 
 F(\frac{\ell-1}{2},1) & =  \sum\limits_{i= \lceil\frac{k}{2}\rceil}^{k}\binom{\frac{\ell-1}{2}}{i}\binom{1}{1}\binom{\frac{\ell-1}{2}}{k-1-i}+ \sum\limits_{i=\lceil\frac{k+1}{2}\rceil}^{k}\binom{\frac{\ell-1}{2}}{i}\binom{1}{0}\binom{\frac{\ell-1}{2}}{k-i}\\
  & =  \sum\limits_{i= \lceil\frac{k}{2}\rceil+1}^{k}\binom{\frac{\ell-1}{2}}{i-1}\binom{\frac{\ell-1}{2}}{k-i}+ \sum\limits_{i=\lceil\frac{k+1}{2}\rceil}^{k}\binom{\frac{\ell-1}{2}}{i}\binom{\frac{\ell-1}{2}}{k-i} \\
  & \geq \sum\limits_{i= \lceil\frac{k}{2}\rceil+1}^{k}\binom{\frac{\ell-1}{2}}{i-1}\binom{\frac{\ell-1}{2}}{k-i}+ \sum\limits_{i=\lceil\frac{k}{2}\rceil+1}^{k}\binom{\frac{\ell-1}{2}}{i}\binom{\frac{\ell-1}{2}}{k-i}.  
 \end{align*}  
On the other hand, by Equality (\ref{eq-F(X,0)}),  $F(\frac{\ell+1}{2},0)= \sum\limits_{i=\lceil\frac{k}{2}\rceil+1}^{k}\binom{\frac{\ell+1}{2}}{i}\binom{\frac{\ell-1}{2}}{k-i}$. Therefore, using equalities $\binom{\frac{\ell+1}{2}}{i}=\binom{\frac{\ell-1}{2}}{i}+\binom{\frac{\ell-1}{2}}{i-1}$, we have 
\begin{equation*} 
 F(\frac{\ell-1}{2},1)-F(\frac{\ell+1}{2},0) \geq \sum\limits_{i= \lceil\frac{k}{2}\rceil+1}^{k}\binom{\frac{\ell-1}{2}}{i-1}\binom{\frac{\ell-1}{2}}{k-i}+ \big(\sum\limits_{i=\lceil\frac{k}{2}\rceil+1}^{k}\binom{\frac{\ell-1}{2}}{i}\binom{\frac{\ell-1}{2}}{k-i}-\sum\limits_{i=\lceil\frac{k}{2}\rceil+1}^{k}\binom{\frac{\ell+1}{2}}{i}\binom{\frac{\ell-1}{2}}{k-i}\big) =0.
\end{equation*}

This completes the proof of Lemma \ref{main-lemma}.

\end{document}